\documentclass[10pt,reqno]{amsart}
\usepackage{graphicx,epsfig}
\usepackage{epstopdf}
\usepackage{pdfpages}
\usepackage{amssymb}
\usepackage{empheq}
\usepackage{cases}
\usepackage{amsthm,amsmath}
\usepackage{caption,lipsum}
\usepackage{stmaryrd}
\usepackage{tabularx}
\usepackage{color}
\usepackage{empheq,float}
\DeclareGraphicsExtensions{.eps}

\newtheorem{theorem}{Theorem}[section]
\newtheorem{lemma}[theorem]{Lemma}

\theoremstyle{definition}

\newtheorem{example}[theorem]{Example}

\theoremstyle{remark}

\newcommand{\fe}{\mathrm{e}}

\newcommand{\bR}{{\mathbb R}}
\newcommand{\bC}{{\mathbb C}}

\newcommand{\bT}{{\mathbb T}}
\newcommand{\bN}{{\mathbb N}}
\newcommand{\bZ}{{\mathbb Z}}

\numberwithin{equation}{section}

\begin{document}

\title[Convergence of schemes for disordered NLS]{Numerical integrators for continuous disordered nonlinear Schr\"odinger equation}

%\author[K. Schratz]{Katharina Schratz}
%\address{\hspace*{-12pt}K.~Schratz: Laboratoire Jacques-Louis Lions, Sorbonne Universit\'e, 75252 Paris, France}
%\email{katharina.schratz@ljll.math.upmc.fr}

\author[X. Zhao]{Xiaofei Zhao}
\address{\hspace*{-12pt}X.~Zhao: School of Mathematics and Statistics \& Computational Sciences Hubei Key Laboratory, Wuhan University, 430072 Wuhan, China}
\email{matzhxf@whu.edu.cn}

%\subjclass[2010]{Primary }
%
%\keywords{Vlasov-Poisson equation, Three dimensions, Strong magnetic field, Varying direction, Uniformly accurate method, Particle-In-Cell}

\date{}

\dedicatory{}

\begin{abstract}
In this paper, we consider the numerical solution of the continuous disordered nonlinear Schr\"odinger equation, which contains a spatial random potential. We address the finite time accuracy order reduction issue of the usual numerical integrators on this problem, which is due to the presence of the random/rough potential. By using the recently proposed low-regularity integrator (LRI) from (33, SIAM J. Numer. Anal., 2019), we show how to integrate the potential term by losing two spatial derivatives. Convergence analysis is done to show that LRI has the second order accuracy in $L^2$-norm for potentials in $H^2$. Numerical experiments are done to verify this theoretical result. More numerical results are presented to investigate the accuracy of LRI compared with classical methods under rougher random potentials from applications.
\end{abstract}

\maketitle

\section{Introduction}
In this work, we are concerned with the following one-dimensional disordered nonlinear Schr\"{o}dinger equation (D-NLS) on a torus or whole space which contains a spatial random potential \cite{debussche3,debussche1,CAL2,Soffernew}:
\begin{equation}\label{model}
 \left\{\begin{split}
& i\partial_tu(x,t)=-\partial_{x}^2 u(x,t)+\xi(x) u(x,t)+\lambda|u(x,t)|^2u(x,t),
 \quad t>0,\ x\in\mathcal{D},\\
 &u(x,0)=u_0(x),\quad x\in\mathcal{D},
 \end{split}\right.
\end{equation}
where $\mathcal{D}=\bR$ or $\bT$ (a torus), $u(x,t):\mathcal{D}\times[0,\infty)\to\bC$ is the unknown, $u_0(x)$ is the given initial data, $\xi(x):\mathcal{D}\to\bR$ is a given random potential, and $\lambda\in\bR$ is a given parameter with $\lambda>0$ for defocusing nonlinear interaction and $\lambda<0$ for focusing case. The D-NLS
(\ref{model}) is of paramount importance in mathematical and physical studies. For example,  mathematically, it can be considered as the complex parabolic Anderson model \cite{debussche3,debussche1}. Physically, it has been considered for modeling the Anderson localization of Bose-Einstein Condensates \cite{pra,PRLnew} and for the study of nonlinear dispersive wave dynamics in a disordered medium \cite{Conti,nonlinear wave,CAL2}. The spatial random potential $\xi(x)$, as considered in many related numerical works and physical applications in the literature, is composed out of a large number of random independent and identically distributed (i.i.d.) Fourier components \cite{sde3,sde1,pre3,CAL2,sde2,pra,PRLnew,whitenoise}.

The D-NLS (\ref{model}) in fact has received many mathematical attentions and we refer to \cite{Chouk,debussche2,debussche3,debussche1,Labbe,CAL1,Gu} for intensive theoretical studies from different aspects.
It has a discrete analogue known as the nonlinear Schr\"{o}dinger equation on the disordered lattice \cite{CAL0,Wang2}:
%For each $x\in\bT$, the potential $\xi(x)$ is taken as a uniform distribution in some interval $[-\omega,\omega]$ with $\omega>0$.
%The discrete version of (\ref{model}), for example in 1D known as the nonlinear Schr\"{o}dinger equation on a disordered lattice:
%\begin{equation}\label{model2}
% \left\{\begin{split}
\begin{equation}\label{discrete}i\dot{u}_l(t)=-J[u_{l+1}(t)+u_{l-1}(t)]
+\xi_lu_l(t)+\lambda|u_l(t)|^2u_l(t),
 \quad t>0,\ l\in\bZ,\end{equation}
% &\psi_l(0)=\psi_l^0,\quad z\in\bZ,
% \end{split}\right.
%\end{equation}
where $J\in\bR$ is a given parameter and $\xi_l$ are a series of i.i.d. random variables. This discrete model has been widely addressed (more than (\ref{model}) so far) in the physical studies, e.g.  \cite{Anderson,review,prl2,prl1,pre1,nature1,pre2}. As a matter of fact, it shares many physical properties \cite{review} with the continuous version (\ref{model}). For the linear regime of (\ref{model}) or (\ref{discrete}), i.e. $\lambda=0$, it is well-understood as the famous Anderson localization \cite{Anderson} which tells that spreadings of any localized waves will be suppressed by the random potential.
For the nonlinear regime $\lambda\neq0$ in (\ref{model}) or (\ref{discrete}),  the competition between the random potential and the nonlinearity makes the long-time dynamics very complicated. For an initially
localized wavepacket in the whole space setup, heuristically, the spreadings in the nonlinear system (\ref{model}) or (\ref{discrete}) could still be suppressed by the random potential, otherwise the nonlinearity would become negligible as time gets large and then the problem goes back to the classical Anderson localization \cite{review}.
The theoretical understandings of this issue are quite challenging and so far the rigours results are still limited \cite{Wang1,review,Soffernew,Wang2}. Therefore, many existing works  rely on the numerical simulations \cite{review,prl2,CAL2,prl1,pre1,nature1,pre2}, and the majority of them are made on the above discrete version (\ref{discrete}).
While, the numerical research outcomes so far conflict with the heuristical argument and unlimited spreadings were observed \cite{prl2,prl1}.

 %However, existing numerical simulations do not support this heuristic argument \cite{prl1,prl2,pre1}.
%One reason could be the chaotic behaviour in the equation (\ref{model}) in large time. The other issue as we will report here for the continuous model is the accuracy of the numerical integrators. This obstacle makes numerical investigations of (\ref{model}) more challenging than the discrete model (\ref{model2}).

From the numerical point of view, there is some significant difference between solving the continuous D-NLS
model (\ref{model}) and the discrete one  (\ref{discrete}).
%(\ref{model}) has been proved to be well-posed in  \cite{debussche1,debussche3,Labbe,Chouk}.
In the existing numerical studies in the literature, e.g. \cite{prl2,prl1,pre1}, some classical symplectic integrators including the splitting schemes and finite difference schemes of the second or the fourth order accuracy were used to solve the discrete model. The high order accuracy and the symplecticity of the schemes on the ODEs could  provide reliable  long-term approximations of (\ref{discrete}). However, it is far from straightforward to extend these numerical works from the discrete model (\ref{discrete}) to the continuous model (\ref{model}). The challenging issues  for solving (\ref{model}) include the domain truncation of whole space problem for long-time computing, the efficient and accurate spatial discretization, the long-time behaviour of the fully discretized integrators on the PDE, and even the accurate temporal approximation up to a fixed finite time, where the last issue is what we shall address in this paper. For attempts to control the temporal approximation error to arbitrary time, we refer to the work \cite{Soffer}. For the studies of the nonlinear Schr\"{o}dinger equations with time-dependent random/rough potentials, we refer to \cite{sde3,debussche2,sde1,sde4,sde2,Schlag}.

For the popular classical numerical integrators such as Strang splitting method or Crank-Nicolson finite difference method, when applied to the continuous D-NLS model (\ref{model}) as we shall see in our numerical experiments, their accuracy order could be much less than the optimal second order rate. This is because the regularity of the spatial noise $\xi(x)$ could be very low (even in negative Sobolev space in some cases \cite{whitenoise2,whitenoise}). This directly illustrates the problem of using high order splitting schemes  for integrating (\ref{model}), that the high order commutators \cite{Lubich,Thalhammer} generate high order  differentiations on $\xi$, and so the unboundedness of such terms causes significant  accuracy  order reduction. For instance, without the  boundedness of $\partial_x^4\xi$, the Strang splitting method will fail to offer the second order accuracy. Moreover, the potential will induce roughness to the solution of (\ref{model}).
 Recently, when the potential $\xi(x)$ is assumed to be a Gaussian white noise,
the well-posedness of the Cauchy problem (\ref{model}) in  torus has been established only in $H^1$ space \cite{debussche3}, and its well-posedness in the whole space case has later been given in  \cite{debussche1}. We will show in our study of (\ref{model}) that for $\xi\in H^2$ and smooth initial data $u_0$, the regularity of the solution is just in $H^4$. Such solution with low regularity also causes order reductions for classical integrators. For instance, the  Crank-Nicolson method requests the boundedness of the six order spatial derivatives of the solution to reach its second order accuracy. For a rigorous analysis of Crank-Nicolson method, we refer to \cite{dnls-fd}. To overcome the order reduction problem due to the lack of regularity in the solution, some state-of-the-art called the low-regularity Fourier integrators (LRI) have been proposed very recently for the nonlinear Schr\"{o}dinger equations without the potential \cite{Schratznew,kath1,lownls,lowNLS2}, where the roughness was introduced through the initial data. The particular efforts of LRI were made to integrate the nonlinearity in physical space by losing two spatial derivatives for the second order accuracy in time. In the contrast to the setup in \cite{Schratznew,kath1,lownls,lowNLS2}, many physical problems related to (\ref{model}) consider a smooth initial input \cite{CAL2,pra,PRLnew,nature1}.
%the local truncation error of the schemes such as the second order splitting schemes \cite{Lubich,lownls} contain the principle error term $O(\tau^3\partial_x^4u)$ (with $\tau$ the time step), and due to the presence of the random potential $\xi(x)$, the solution of (\ref{model}) is not smooth enough even under smooth initial profile  $u_0$. It is known that the regularity of the spatial noise could be very low  \cite{whitenoise,whitenoise2}.

The purpose of this work is to develop the LRI method for solving the D-NLS (\ref{model}), and we shall consider the random potential $\xi(x)$ generated either by the Gaussian distribution or uniform distribution with regularity at most $H^2$ or $L^2$.  We shall first review some classical methods and address their order reductions on (\ref{model}) with more details. Then
to raise the accuracy order, we consider the LRI method from \cite{kath1} and extend it to integrate  (\ref{model}). Special efforts are made here to integrate the potential term in (\ref{model}) up to the second order accuracy by losing two spatial derivatives, and to keep the final form of the scheme explicitly defined in the physical space so that it can be programmed efficiently under the Fourier pseudo-spectral method for spatial discretization.
Meanwhile, efforts are also made to avoid CFL-type stability conditions in the scheme, which is very important since the roughness in space requires very small spatial mesh size for accurate approximations. We will prove the quadratic convergence in time of the proposed LRI in $L^2$-norm for $\xi\in H^2$. For rougher potential case, e.g. $\xi\in L^2$, many numerical experiments are done to show that the expectation of the $L^2$-error of LRI  converges in time at a rate between one and two. Comparisons are made with the classical methods.

The rest of the paper is organised as follows. In Section \ref{sec:method}, we review some classical integrators for nonlinear Schr\"{o}dinger equations. In Section \ref{sec:lri}, we derive the LRI scheme for solving (\ref{model}) with convergence analysis in the case $\xi\in H^2$. Numerical experiments with general random potentials are given in Section \ref{sec:result}, and some conclusions are drawn in Section \ref{sec:con}.

%By this study, we hope to provide a better understanding for numerical simulations of (\ref{model}).

\section{Classical numerical methods}\label{sec:method}
In this section, we will briefly review some popular numerical integrators including the splitting methods and the finite difference methods for solving the nonlinear Schr\"odinger equations. In fact, these methods have been considered to integrate the discrete model (\ref{discrete}) in the physical literatures, and here we present them for the continuous D-NLS    (\ref{model}).

For the computational reason, we would always consider the torus setup of  (\ref{model}), i.e. $x\in\bT=(-L,L)$. By choosing the domain size $L>0$ large enough, this is also an approximation of the whole space case when the waves are away from the boundary during the dynamics. We denote $\tau=\Delta t>0$ as the time step, $t_n=n\tau$ as the time grids and $u^n=u^n(x)\approx u(x,t_n)$ for $n\geq0$ with $u^0=u_0$. For simplicity, to present the schemes and some known convergence results, let us consider  a sampled or deterministic $\xi(x)$ that could potentially be very rough.

\subsection{Splitting scheme} The time-splitting type schemes \cite{Splitting} have been widely used to solve nonlinear Schr\"{o}dinger equations \cite{BaoCai,Bao,Faou,Lubich}.
The method is to split (\ref{model}) into two subproblems:
\begin{equation*}%\label{sub1}
\Phi_T^t:\quad \left\{
\begin{split}
&i\partial_sv(x,t)=-\partial_{x}^2 v (x,t),\quad t>0,\ x\in\bT,\\
&v(x,0)=v_0(x),\quad x\in\bT,
\end{split}\right.
\end{equation*}
and
\begin{equation*}%\label{sub2}
\Phi_V^t:\quad\left\{
\begin{split}
&i\partial_sw(x,t)=\left(\xi(x)+\lambda|w(x,t)|^2\right)w(x,t),\quad t>0,\ x\in\bT,\\
&w(x,0)=w_0(x),\quad x\in\bT.
\end{split}\right.
\end{equation*}
For both of the subproblems, we have the exact solutions:
$$v(x,t)=\Phi_T^t(v_0(x))=\fe^{it\partial_{x}^2}v_0(x),\qquad
w(x,t)=\Phi_V^t(w_0(x))=\fe^{-it(\xi(x)+\lambda|w_0(x)|^2)}w_0(x).$$
Then through composition, the classical Lie-Trotter splitting method for (\ref{model}) reads:
\begin{eqnarray*}
u(x,t_{n+1})\approx\Phi_V^{\tau}\circ\Phi_T^{\tau}(u(x,t_{n})),\quad n\geq0,
\end{eqnarray*}
which is of the first order accuracy in time under usual smooth enough setup \cite{Besse}. By further  requesting more smoothness on the initial data and the potential in (\ref{model}), with more times of composition, one could in principle get the high order accurate splitting methods \cite{Besse,SplittingSINUM,Hairer,Splitting,lownls,Lubich,Thalhammer}. Among the second order methods, the most popular one is the Strang splitting method:
\begin{eqnarray}\label{strang}
u^{n+1}(x)=\Phi_V^{\tau/2}\circ\Phi_T^{\tau}\circ\Phi_V^{\tau/2}(u^n(x)),
\quad n\geq0,\ x\in\bT.
\end{eqnarray}
The scheme is time-symmetric, symplectic and it preserves the mass of the Schr\"{o}dinger equation \cite{Bao,Faou}.

It is well known that the local truncation error of Strang splitting (\ref{strang}) is determined by the double Lie commutator of the kinetic part and the potential part of the equation (\ref{model}), i.e. with notations from \cite{Lubich},
$$local\ error =O\left(\tau^3[\hat{T},[\hat{T},\hat{V}]](u)\right),\quad \mbox{with}\quad \hat{T}(u)=\partial_x^2 u,\quad \hat{V}(u)=(\xi+\lambda|u|^2)u.$$
That is to say the truncation error of Strang splitting for (\ref{model}) contains the terms $O(\tau^3\partial_x^4u)$ and $O(\tau^3\partial_x^4\xi)$. In order to reach the optimal  accuracy of the Strang splitting scheme (\ref{strang}), one will then ask for enough smoothness on both the potential $\xi$ and the initial data $u_0$ in (\ref{model}). By using the analysis from \cite{Lubich} without essential modifications, for a given $\xi(x)$, the global approximation error of Strang splitting scheme (\ref{strang}) for solving (\ref{model}) up to a fixed final time $T>0$ reads
$$\|u-u^n\|_{L^2}=O(\tau^2),\quad 0\leq n\leq T/\tau,\quad \mbox{if}\quad \xi, u_0\in H^4(\bT).$$
For higher order splitting schemes and their convergence results under smooth potential and initial data case, we refer to \cite{Thalhammer}.

%Here, we consider another second order symplectic splitting method which is known as the SBAB2 method from \cite{SBAB}:
%\begin{eqnarray}\label{SBAB}
%u^{n+1}(x)=\Phi_V^{\tau/6}\circ\Phi_T^{\tau/2}\circ\Phi_V^{2\tau/3}
%\circ\Phi_T^{\tau/2}\circ\Phi_V^{\tau/6}(u^n(x)),
%\quad n\geq0,\ x\in\bT.
%\end{eqnarray}
The Strang splitting scheme (\ref{strang}) has some analogies such as the SBAB method \cite{SBAB}, which are also second order accurate under smooth setup.
These splitting methods have been applied to integrate the discrete model (\ref{discrete}) with uniform distribution random potential on a very large time interval in \cite{prl2,prl1,pre1}, and unlimited subdiffusions of initially localized waves were observed. Numerically, the performance of the SBAB method is similar to the Strang splitting, so here we just consider the Strang splitting as a representative from this class of methods. The  spatial discretizations of (\ref{strang})  can be made by the Fourier
pseudo-spectral method \cite{Shen} thanks to the periodic setup.

\subsection{Finite difference method} As the most traditional numerical  methods, different kinds of finite difference discretizations have been proposed for the nonlinear Schr\"{o}dinger equations \cite{AkrivisFD,BaoCai,dnls-fd,Sanz-SernaFD,WTC}. Here, we consider a second order semi-implicit finite difference integrator (FD) for (\ref{model}):
\begin{align}\label{FD}
&i\frac{u^{n+1}(x)-u^{n-1}(x)}{2\tau}=-\frac{1}{2}\partial_{x}^2
\left[u^{n+1}(x)+u^{n-1}(x)\right]
+\xi (x)u^{n}(x)+\lambda|u^n(x)|^2u^n(x),
\quad n\geq1,\ x\in\bT,\end{align}
with the starting value obtained from a first order version:
\begin{align*}
&i\frac{u^{1}(x)-u^{0}(x)}{\tau}=-\partial_{x}^2
u^{1}(x)
+\xi (x)u^{0}(x)+\lambda|u^0(x)|^2u^0(x),\quad x\in\bT.
\end{align*}
The implicit and explicit treatments in the above scheme are to avoid CFL-type  stability conditions and meanwhile to make it efficient for programming  under the Fourier pseudo-spectral method for spatial discretizations.

This FD method has also been applied to solve the discrete model (\ref{discrete}) in \cite{prl2,prl1,pre1}, where the same long-term  dynamics were observed as from the splitting schemes. By Taylor's expansion, it is clear that
the local truncation error in (\ref{FD}) depends on $O(\tau^3\partial_{t}^3u)$ and $O(\tau^3\partial_x^2\partial_t^2u)$. The boundedness of these two terms, under the smooth potential case, is equivalent to $u\in H^6$ by the equation (\ref{model}). So as has been established in \cite{AkrivisFD,BaoCai1,dnls-fd,Sanz-SernaFD,WTC}, the global error of the FD method (\ref{FD}) for solving (\ref{model}) up to a fixed $T>0$ is
$$\|u-u^n\|_{L^2}=O(\tau^2),\quad 0\leq n\leq T/\tau, \quad \mbox{if}\quad \xi\in C^\infty(\bT),\ u\in H^6(\bT).$$
For the rough potential case of $\xi\in L^\infty(\bT)$, a rigorous first order convergence result has been established in \cite{dnls-fd}.
%Some eighth-order Runge-Kutta method was considered in \cite{pre1}.

We will apply the above presented second order classical methods to integrate the continuous D-NLS model (\ref{model}), and study in particular their practical accuracy order in Section \ref{sec:result} by numerical experiments. As we shall see, all of them will suffer from accuracy order reductions, and so this gives us no motivation to include higher order methods here for (\ref{model}), although some fourth order or eighth-order extensions of the presented second order scheme (\ref{strang}) have also been considered in the physical studies of the discrete problem (\ref{discrete}) in \cite{pre1}.

As can be seen from above, the convergence order reduction problem is essentially due to the lack of enough regularity in the potential $\xi$ as well as in the solution $u$ of D-NLS (\ref{model}). In some numerical recent work \cite{kath1, lownls} on nonlinear Schr\"odinger equations, the low-regularity of the solution and the order reduction problem of classical methods have also been addressed, where the roughness was introduced to the  equation through the initial data. Here in D-NLS (\ref{model}), the difference is that the roughness is induced inherently to the solution $u$ by the spatial random potential $\xi$, even under smooth initial input $u_0$ as we shall discuss in details later. Therefore, numerically more special efforts are needed to take care of both the potential and the solution.

\section{Low-regularity Fourier integrator}\label{sec:lri}
To raise the accuracy of temporal approximation, in this section,
we will consider the low-regularity Fourier integrator (LRI) from \cite{kath1} to solve the D-NLS (\ref{model}). The LRI has been proposed originally to integrate the cubic nonlinearity in the equation with the second order accuracy  by losing two spatial derivatives of the solution. Here, we need some particular efforts to handle the potential term in (\ref{model}).
We shall first derive the scheme and then analyze its convergence.

We shall follow some of the notations used in \cite{kath1}, and for simplicity, the spatial variable $x$ of the functions will be omitted, e.g.  $u(t)=u(x,t)$ and $\xi=\xi(x)$, $x\in\bT$.
\subsection{Derivation of scheme}
By introducing the twisted variable,
\begin{equation}\label{v def}
v(t)=\fe^{-it\partial_{x}^2}u(t),\quad t\geq0,\ x\in\bT,
\end{equation}
(\ref{model}) becomes
$$i\partial_tv(t)=\fe^{-it\partial_x^2}\left[\xi\fe^{it\partial_x^2}v(t)
+\lambda|\fe^{it\partial_x^2}v(t)|^2\fe^{it\partial_x^2}v(t)\right],\quad t>0,\ x\in\bT,$$
and then by using the Duhamel's formula, the mild of solution of (\ref{model}) at $t=t_n+s$ for some $n\geq0$ and $ 0\leq s\leq\tau$ reads
\begin{align}
 &v(t_{n}+s)\label{vduhamel}\\
=&v(t_n)-i\fe^{-it_n\partial_{x}^2}\int_0^s\fe^{-i\rho\partial_{x}^2}
  \left[\xi\fe^{i(t_n+\rho)\partial_{x}^2}v(t_n+\rho)
  +\lambda|\fe^{i(t_n+\rho)\partial_{x}^2}v(t_n+\rho)|^2\fe^{i(t_n+\rho)
  \partial_{x}^2}v(t_n+\rho)\right]d\rho.\nonumber
\end{align}
By denoting for simplicity
\begin{equation}\label{notation Vn} V_n(s)=\fe^{it_n\partial_{x}^2}v(t_n+s),\quad 0\leq s\leq \tau,
\end{equation}
 we get from (\ref{vduhamel}),
\begin{align}\label{vduhamel1}
  V_n(s)=V_n(0)-i\int_0^s\fe^{-i\rho\partial_{x}^2}
  \left[\xi\fe^{i\rho\partial_{x}^2}V_n(\rho)
  +\lambda|\fe^{i\rho\partial_{x}^2}V_n(\rho)|^2
  \fe^{i\rho\partial_{x}^2}V_n(\rho)\right]d\rho,\quad 0\leq s\leq\tau.
\end{align}
It is clear from the above that $V_n(s)=V_n(0)+O(s)$, where the term $O(s)$ does not involve any spatial derivatives.
Therefore, by denoting $V_n=V_n(0)$ for short and making approximations  $V_n(\rho)\approx V_n$ and $\fe^{\pm i\rho\partial_x^2}\approx1$ in the integrand of (\ref{vduhamel1}), where the latter clearly costs two spatial derivatives, we get the following expansion/approximation:
\begin{equation}\label{Vn 1}
V_n(\rho)= V_n-i\rho\left(\xi V_n+\lambda|V_n|^2V_n\right)+O(\tau^2),\quad 0\leq \rho\leq\tau,\ n\geq0.
\end{equation}
The constant in the above truncation term $O(\tau^2)$ depends on $\partial_x^2u$ and $\partial_x^2\xi$. Under the boundedness of these terms, globally, its optimal accuracy is at  first order.

Now we go for the second order accurate approximation. By plugging the first order accurate  approximation (\ref{Vn 1}) into (\ref{vduhamel1}), and then letting $s=\tau$, we find
\begin{align}\label{vduhamel2}
  V_n(\tau)=&
  V_n+A_1(V_n)+A_2(V_n)+\mathcal{N}(V_n)+O(\tau^3),\quad n\geq0,
\end{align}
where $\mathcal{N}$ denotes the $O(\tau)$ and $O(\tau^2)$ terms coming out from the cubic nonlinearity:
\begin{align}
 \mathcal{N}(V_n)=&-i\lambda\int_0^\tau\fe^{-i\rho\partial_x^2}
 \left[\left(\fe^{-i\rho\partial_{x}^2}\overline{V_n}\right)
 \left(\fe^{i\rho\partial_{x}^2}V_n\right)^2\right]d\rho\nonumber\\
 &+\lambda^2\int_0^\tau \rho\fe^{-i\rho\partial_{x}^2}\left[
\left(\fe^{-i\rho\partial_{x}^2}\left(|V_n|^2\overline{V_n}\right)\right)
\left(\fe^{i\rho\partial_{x}^2}V_n\right)^2\right]d\rho\nonumber\\
&-2\lambda^2\int_0^\tau \rho\fe^{-i\rho\partial_{x}^2}
\left[\left(\fe^{i\rho\partial_{x}^2}\left(|V_n|^2V_n\right)\right)
\left|\fe^{i\rho\partial_{x}^2}
V_n\right|^2\right]d\rho,\quad n\geq0,\label{Nn def}
\end{align}
and $A_1(V_n),\,A_2(V_n)$ denote the terms coming out from  the potential part:
\begin{subequations}\label{A12 def}
\begin{align}
 A_1(V_n)
 =&\lambda\int_0^\tau s\fe^{-is\partial_x^2}
 \left[\left(\fe^{-is\partial_x^2}(\xi\overline{V_n})\right)
 \left(\fe^{is\partial_x^2}V_n\right)^2\right]ds-2\lambda\int_0^\tau
 s\fe^{-is\partial_x^2}\left[\left|\fe^{is\partial_x^2}V_n\right|^2
 \left(\fe^{is\partial_x^2}(\xi V_n)\right)\right]ds\nonumber\\
 &-\int_0^\tau
 s\fe^{-is\partial_x^2}\left[\xi\fe^{is\partial_x^2}\left(\xi V_n+\lambda|V_n|^2V_n\right)\right]ds,\label{A1n def}\\
 A_2(V_n)
 =&-i\int_0^\tau
 \fe^{-is\partial_x^2}\left(\xi\fe^{is\partial_x^2}V_n\right)ds,\quad n\geq0.\label{A2n def}
\end{align}
\end{subequations}
In (\ref{vduhamel2}), the truncation terms are those $O(\tau^3)$ terms  from the multiplications of the cubic nonlinearity and the potential term, where the highest order spatial derivatives evolved in the error constant are still $\partial_x^2u$ and $\partial_x^2\xi$.

In the following,
we shall compute or further approximate the three terms $A_1(V_n),\,A_2(V_n)$ and $\mathcal{N}(V_n)$ in a sequel. The goal is to obtain their explicit (approximated) formulas in the physical space and meanwhile reach the global second order accuracy by losing at most two spatial derivatives  of the functions including the solution $u$ and the potential $\xi$.  Firstly,
let us denote the Fourier transform of some function $f(x):\bT=(-L,L)\to\bC$ as
$$\widehat{f}_l=\displaystyle\frac{1}{2L}\int_\bT \fe^{-i\mu_l(x+L)}f(x)\,dx,
\quad \mbox{with}\quad \mu_l=\frac{\pi l}{L},$$
and we define the operator $\partial_x^{-m}$ as
$$
\partial_x^{-m} f(x)=\sum_{l\neq0}
(i\mu_l)^{-m}\widehat{f}_l\fe^{i\mu_l(x+L)},\quad m\in\bN_+.
$$

\textbf{Integrating potential terms.}

For $A_1(V_n)$, by directly letting $\fe^{\pm is\partial_x^2}\approx1$ in (\ref{A1n def}) and then integrating the rests exactly, it is easy to see that the following  approximation
\begin{equation}\label{A1 app}
A_1(V_n)\approx-\lambda\tau^2
\xi|V_n|^2V_n-\frac{\tau^2}{2}
\xi^2V_n,\quad n\geq0,
\end{equation}
fulfills our aforementioned goal. The truncation error here is $O(\tau^3)$ with two spatial derivatives involved in the error constant.
Now the difficulty comes to compute $A_2(V_n)$. First of all, we write $A_2(V_n)$ in the Fourier space
\begin{align}
A_2(V_n)%=-i\int_0^\tau
%\fe^{-is\partial_x^2}\left(\xi\fe^{is\partial_x^2}V_n\right)ds
=&-i\sum_{l_1+l_2=l}\fe^{i\mu_l(x+L)}\int_0^\tau \fe^{is(\mu_l^2-\mu_{l_2}^2)}\widehat{\xi}_{l_1}\widehat{(V_n)}_{l_2}ds
 \nonumber\\
 =&-i\sum_{l_1+l_2=l}\fe^{i\mu_l(x+L)}\int_0^\tau \fe^{is(\mu_{l_1}^2+2\mu_{l_1}\mu_{l_2})}
 \widehat{\xi}_{l_1}\widehat{(V_n)}_{l_2}ds.\label{A2 def}
 \end{align}
 It is seen that the two frequencies $l_1$ and $l_2$ are coupled here, which prevents us from evaluating the integral explicitly in the physical space. So some further approximations are needed, otherwise the scheme will be defined in Fourier space and one will have to deal with the convolution.
Noting that
\begin{equation}\label{sec2 eq0}\fe^{is(\mu_{l_1}^2+2\mu_{l_1}\mu_{l_2})}=\fe^{is\mu_{l_1}^2}\fe^{2is
\mu_{l_1}\mu_{l_2}}
=\fe^{is\mu_{l_1}^2}\left(1+2i\mu_{l_1}\mu_{l_2}s\right)
+O(s^2\mu_{l_1}^2\mu_{l_2}^2),
\end{equation}
 we can approximate the integrand in (\ref{A2 def}) as
\begin{eqnarray}
&&\quad\int_0^\tau\fe^{is(\mu_{l_1}^2+2\mu_{l_1}\mu_{l_2})}ds\nonumber\\
&&\approx\int_0^\tau\fe^{is\mu_{l_1}^2}\left(1+2i\mu_{l_1}\mu_{l_2}s\right)ds
=\left\{\begin{split}&\frac{\fe^{i\tau \mu_{l_1}^2}-1}{i\mu_{l_1}^2}+\frac{2\mu_{l_2}}{\mu_{l_1}^3}\left[\fe^{i\tau\mu_{l_1}^2}
(i+\tau\mu_{l_1}^2)-i\right],\quad l_1\neq0,\\
&\tau,\qquad\qquad\qquad\qquad\qquad\qquad\qquad\qquad\quad\ \, l_1=0,\end{split}\right.\qquad\qquad\label{tofilter}
\end{eqnarray}
so that the truncation error here is $O(\tau^3)$ with the loss of two spatial derivatives. It seems that this approximation fulfills our goal, but in the above formula (\ref{tofilter}) for $l_1\neq0$, we note that the second term is proportional to $\mu_{l_2}$, which corresponds to an explicit spatial derivative in the physical space. This kind of explicit spatial derivative term  will induce stability issue in the final numerical scheme, and one will end up with a CFL constraint on the time step and spatial mesh size for computing. It is known that spatial discretizations such as the Fourier pseudo-spectral method will only have an algebraic convergence rate and the rate  could be very low when the solution is less regular \cite{Shen}. Therefore, one will need to use small mesh size in space for accurate approximations, and then  the CFL constraint would lead to severe inefficiency for practical computations in the rough solution situation.  To enhance the stability, we introduce a filter \cite{BCZ}:
$$\frac{\sin(\mu_{l_2}\tau)}{\tau}=
\mu_{l_2}+O(\tau\mu_{l_2}^2),$$
 and the approximation in (\ref{tofilter}) further becomes
\begin{align}\label{sec2 eq1}
 \frac{2\mu_{l_2}}{\mu_{l_1}^3}\left[\fe^{i\tau \mu_{l_1}^2}
(i+\tau \mu_{l_1}^2)-i\right]
=\frac{2\sin(\mu_{l_2}\tau)}{\mu_{l_1}^3\tau}
\left[\fe^{i\tau\mu_{l_1}^2}
(i+\tau \mu_{l_1}^2)-i\right]+O(\tau^3\mu_{l_1} \mu_{l_2}^2),\quad l_1\neq0,
\end{align}
where we used the fact $\fe^{i\tau\mu_{l_1}^2}
(i+\tau \mu_{l_1}^2)-i=O(\tau^2\mu_{l_1}^4)$. Similar filtering strategy for stability has been proposed in a very recent work \cite{Schratznew} for the nonlinear Schr\"odinger equation, and we remark here that there are other possible choices of the filters \cite{Hairer} that could work here. Finding out the best one is beyond the scope of this work.
Then in total, (\ref{A2 def}) is approximated as
\begin{align}
  A_2(V_n)
  \approx&-i\sum_{l_1+l_2=l,l_1\neq0}\fe^{i\mu_l(x+L)}
  \left[\frac{\fe^{i\tau \mu_{l_1}^2}-1}{i\mu_{l_1}^2}+
  \frac{2\sin(\mu_{l_2}\tau)}{\mu_{l_1}^3\tau}\left(\fe^{i\tau \mu_{l_1}^2}
(i+\tau \mu_{l_1}^2)-i\right)\right]\widehat{\xi}_{l_1}\widehat{(V_n)}_{l_2}\nonumber\\
&-i\tau\sum_{l_2}\fe^{i\mu_{l_2}(x+L)}\widehat{\xi}_{0}\widehat{(V_n)}_{l_2}\nonumber\\
=&\xi_1
V_n
-2
\xi_2\sin(-i\tau\partial_x)V_n-i\tau\widehat{\xi}_0V_n,\quad n\geq0,\label{A2 app}
\end{align}
where
\begin{equation}\label{xi12}
\xi_1:=\frac{\fe^{-i\tau\partial_x^2}-1}{\partial_x^2}\xi,
\qquad \xi_2:=\frac{\fe^{-i\tau\partial_x^2}
(i-\tau\partial_x^2)-i}{\tau\partial_x^3}
\xi,\qquad x\in\bT.\end{equation}
Here the $\xi_1$ and $\xi_2$ can be regarded as two regularized versions of the rough potential $\xi$, which can be pre-computed. The truncation error here in (\ref{A2 app}) is a combination of error from (\ref{sec2 eq0}) and (\ref{sec2 eq1}), which in total is $O(\tau^3)$ with the loss of two spatial derivatives.
Now the formula (\ref{A2 app}) is defined explicitly in the physical space, and it can be programmed efficiently under the fast Fourier transform (FFT).

\textbf{Integrating the nonlinearity.}

For the part $\mathcal{N}(V_n)$ defined in (\ref{Nn def}), we completely follow the second order approximation strategy in \cite{kath1}. To be consistent, let us briefly go through some key approximations.  First of all, by denoting
$$J(V_n)=-i\lambda\int_0^\tau\fe^{-i\rho\partial_x^2}
\left[\left(\fe^{-i\rho\partial_x^2}\overline{V_n}\right)
\left(\fe^{i\rho\partial_x^2}V_n\right)^2\right]d\rho,\quad n\geq0,$$
the nonlinear terms in $\mathcal{N}(V_n)$ can be written as
\begin{align*}
\mathcal{N}(V_n)=&J(V_n)
+\lambda^2\int_0^\tau \rho\fe^{-i\rho\partial_x^2}\left[
\left(\fe^{-i\rho\partial_x^2}(|V_n|^2\overline{V_n})\right)
\left(\fe^{i\rho\partial_x^2}V_n\right)^2\right]d\rho\\
&-2\lambda^2\int_0^\tau \rho\fe^{-i\rho\partial_x^2}
\left[\left(\fe^{i\rho\partial_x^2}(|V_n|^2V_n)\right)
\left|\fe^{i\rho\partial_x^2}
V_n\right|^2\right]d\rho.
\end{align*}
The last two terms of the above are of $O(\tau^2)$, and we can adopt the approximation $\fe^{\pm i\rho\partial_x^2}\approx1$ similarly as before to get
\begin{align}\label{N app}
\mathcal{N}(V_n)\approx &J(V_n)
-\frac{1}{2}\lambda^2\tau^2|V_n|^4V_n.
\end{align}
For the part $J(V_n)$, in Fourier space we have
\begin{align*}
J(V_n)=&-i\lambda\int_0^\tau
\sum_{l_1+l_2+l_3=l}\fe^{i\mu_l(x+L)}
\fe^{i\rho\left[\mu_l^2+\mu_{l_1}^2-(\mu_{l_2}+\mu_{l_3})^2\right]}\widehat{(\overline{V_n})}_{l_1}
\widehat{(V_n)}_{l_2}\widehat{(V_n)}_{l_3}d\rho\\
=&-i\lambda\int_0^\tau
\sum_{l_1+l_2+l_3=l}\fe^{i\mu_l(x+L)}
\fe^{i\rho\left(2\mu_{l_1}\mu_l+2\mu_{l_2}\mu_{l_3}\right)}\widehat{(\overline{V_n})}_{l_1}
\widehat{(V_n)}_{l_2}\widehat{(V_n)}_{l_3}d\rho.
\end{align*}
Then by the key fact for the phase term in the above \cite{kath1}:
$$\fe^{i\rho\left(2\mu_{l_1}\mu_l+2\mu_{l_2}\mu_{l_3}\right)}
= \fe^{2i\mu_{l_1}\mu_l\rho}+\fe^{2i\mu_{l_2}\mu_{l_3}\rho}-1+O(\rho^2\mu_{l_1}\mu_{l_2}
\mu_{l_3}\mu_l),$$
one finds that
\begin{align}
J(V_n)
\approx& J_1(V_n)+J_2(V_n)
+i\lambda\int_0^\tau
\sum_{l_1+l_2+l_3=l}\fe^{i\mu_l(x+L)}\widehat{(\overline{V_n})}_{l_1}
\widehat{(V_n)}_{l_2}\widehat{(V_n)}_{l_3}d\rho\nonumber\\
=&J_1(V_n)+J_2(V_n)+i\lambda\tau|V_n|^2V_n,\label{Jx app}
\end{align}
with
\begin{align*}
 J_1(V_n):=&-i\lambda\int_0^\tau
\sum_{l_1+l_2+l_3=l}\fe^{i\mu_l(x+L)}
\fe^{2i\mu_{l_1}\mu_l\rho}\widehat{(\overline{V_n})}_{l_1}
\widehat{(V_n)}_{l_2}\widehat{(V_n)}_{l_3}d\rho,\\
J_2(V_n):=&-i\lambda\int_0^\tau
\sum_{l_1+l_2+l_3=l}\fe^{i\mu_l(x+L)}
\fe^{2i\mu_{l_2}\mu_{l_3}\rho}\widehat{(\overline{V_n})}_{l_1}
\widehat{(V_n)}_{l_2}\widehat{(V_n)}_{l_3}d\rho.
\end{align*}
It can be seen that the above approximation is $O(\tau^3)$ with the cost of two spatial derivatives.
After exact integrations in the above, the explicit formulas of $J_1(V_n)$ and $J_2(V_n)$ can be obtained in physical space:
\begin{align}
 J_1(V_n)=&\frac{\lambda}{2}\left[\fe^{-i\tau\partial_x^2}
 \partial_x^{-1}\left(\left(\fe^{-i\tau\partial_x^2}
 \partial_x^{-1}\overline{V_n}\right)
 \left(\fe^{i\tau\partial_x^2}V_n^2\right)
 \right)-\partial_x^{-1}\left(\left(\partial_x^{-1}\overline{V_n}\right)V_n^2
 \right)
 \right]+\tau \widehat{(\overline{V_n})}_0\left(V_n^2-
 \widehat{(V_n^2)}_0\right)\nonumber\\
 &+\tau\widehat{(|V_n|^2V_n)}_0,\label{J1}
\end{align}
and
\begin{align}
 J_2(V_n)=\frac{\lambda}{2}
 \left[\fe^{-i\tau\partial_x^2}
 \left(\partial_x^{-1}\fe^{i\tau\partial_x^2}V_n\right)^2
 -\left(\partial_x^{-1}V_n\right)^2\right]\overline{V_n}
 +\tau \widehat{(V_n)}_0\left(2V_n-\widehat{(V_n)}_0\right)
 \overline{V_n}.\label{J2}
\end{align}
Here we refer the readers to \cite{kath1}  for the detailed calculations of the last two formulas above. With the findings (\ref{J1}) and (\ref{J2}), we get the total approximation of $\mathcal{N}(V_n)$ in physical space from (\ref{Jx app}) and (\ref{N app}) as
$$\mathcal{N}(V_n)\approx J_1(V_n)+J_2(V_n)+i\lambda\tau|V_n|^2V_n
-\frac{1}{2}\lambda^2\tau^2|V_n|^4V_n,\quad n\geq0.$$

\textbf{Full scheme.}

Plugging the above approximation for $\mathcal{N}(V_n)$  into (\ref{vduhamel2}), we get
\begin{align*}V_n(\tau)\approx &V_n+i\lambda\tau|V_n|^2V_n
-\frac{1}{2}\lambda^2\tau^2|V_n|^4V_n
+A_1(V_n)+A_2(V_n)+J_1(V_n)+J_2(V_n)\\
=&\left(1+i\lambda\tau|V_n|^2
-\frac{1}{2}\lambda^2\tau^2|V_n|^4\right)V_n
+A_1(V_n)+A_2(V_n)+J_1(V_n)+J_2(V_n).
\end{align*}
As suggested in \cite{kath1}, the first term in the above approximation could be simplified as
\begin{align*}V_n(\tau)\approx
&\fe^{i\lambda\tau|V_n|^2}V_n+A_1(V_n)+A_2(V_n)
+J_1(V_n)+J_2(V_n)+O(\tau^3),
\end{align*}
which gives the desired second order approximation of $V_n(\tau)$ for $n\geq0$ with the cost of two spatial derivatives.

Noticing the fact $u(t_n)=V_n(0)=V_n$, the final second order low-regularity Fourier integrator (LRI) for the D-NLS (\ref{model}) reads: denoting $u^n=u^n(x)$ as the numerical approximation of $u(t_n)$ for $n\geq0$ and taking $u^0=u_0$ in (\ref{model}), then
\begin{align}\label{vduhamel3}
  u^{n+1}=\fe^{i\tau\partial_x^2}
 \Bigg[& \fe^{i\lambda\tau|u^n|^2}
  u^n+J_{1}(u^n)
  +J_{2}(u^n)-\lambda\tau^2
\xi|u^n|^2u^n-\frac{\tau^2}{2}
\xi^2u^n+\xi_1u^n\\
&-2\xi_2\sin(-i\tau\partial_x)u^n-i\tau\widehat{\xi}_0u^n
\Bigg],\quad x\in\bT,\ n\geq0,\nonumber
\end{align}
where $\xi_1,\,\xi_2$ are defined in (\ref{xi12}), and $J_1(\cdot),\,J_2(\cdot)$ are defined in (\ref{J1}) and (\ref{J2}).

The LRI scheme (\ref{vduhamel3}) is fully explicit in time, and it is explicitly defined in physical space. The two regularized potentials $\xi_1$ and $\xi_2$ in (\ref{xi12}) can be computed once for all before the temporal iteration. The spatial discretization in (\ref{vduhamel3}) can be done by the Fourier pseudo-spectral method \cite{Shen} under our periodic setup. In practical implementation, we take $h=|\bT|/N>0$ as the spatial mesh size with some integer $N>0$ as the total number of space grid points, and (\ref{vduhamel3}) can be programmed efficiently via FFT with the computational cost per time step $O(N\log N)$.

\subsection{Convergence analysis}\label{sec:thm} Here we give a convergence analysis for the proposed LRI (\ref{vduhamel3}) under assumption that in (\ref{model}) a sampled potential $\xi\in H^{\gamma+2}(\bT)$ for some $\gamma\geq0$. For simplicity of notations, we shall denote $A\lesssim B$ for $A\leq CB$ with some generic constant $C>0$ independent of $n$ or $\tau$.

\begin{theorem}\label{thm:nls}
Assume that $\xi\in H^{\gamma+2}(\bT)$ and $u_0\in H^{\gamma+2}(\bT)$ for some $\gamma\geq0$, then for the numerical solution $u^n$ defined in the LRI (\ref{vduhamel3}) up to some fixed $T>0$, there exists constant $\tau_0>0$ such that when $0<\tau\leq\tau_0$,
$$\|u(\cdot,t_n)-u^n\|_{H^\gamma}\lesssim \tau^2,\quad 0\leq n\leq T/\tau.$$
\end{theorem}

%Since $V_n(\tau)=\fe^{-i\tau\partial_x^2}u(t_{n+1})$ and $V_n=V_n(0)=u(t_n)$, it is equivalent to prove the theorem on the twisted variable as $\|V_n-V^n\|_{L^2}\lesssim \tau^2$ where we denote $V^n=$
To prove the theorem, we first establish some lemmas for stability and local error estimates.
First of all, we have some formal \emph{a prior} estimate results for the solution of (\ref{model}) when $\xi\in H^{\gamma+2}$.
\begin{lemma}\label{lm1} (A prior estimate in $H^{\gamma+2}$)
If  $\xi\in H^{\gamma+2}(\bT)$ and $u_0\in H^{\gamma+2}(\bT)$ for some $\gamma\geq0$ in (\ref{model}), then for some $T>0$, we have $u\in L^\infty((0,T),H^{\gamma+2}(\bT))$.
\end{lemma}
\begin{proof}
  By Duhammel's formula of (\ref{model}), we have
  \begin{align}\label{lm1 eq0}
    u(t)=\fe^{it\partial_x^2}u_0-i\int_0^t\fe^{i(t-s)\partial_x^2}
    \left(\xi+\lambda|u(s)|^2\right)u(s)ds, \quad t\geq0.
  \end{align}
  By the Bootstrap, we first assume that
  $u\in L^\infty((0,T),H^{\gamma+2}(\bT))$ for some $T>0$. By taking $H^{\gamma+2}$-norm on both sides of (\ref{lm1 eq0}), and using the triangle inequality as well as the Sobolev's inequality, we have
 \begin{align}\label{lm1 eq11}
 \|u(t)\|_{H^{\gamma+2}}\leq& \|u_0\|_{H^{\gamma+2}}+
   \int_0^t\left(\|\xi\|_{H^{\gamma+2}}\|
   u(s)\|_{H^{\gamma+2}}+\lambda\|
   u(s)\|_{H^{\gamma+2}}^3\right)ds,\nonumber\\
  \leq& \|u_0\|_{H^{\gamma+2}}+
  t\left(\|\xi\|_{H^{\gamma+2}}\|
   u\|_{L^\infty((0,T),H^{\gamma+2})}+\lambda\|
   u\|_{L^\infty((0,T),H^{\gamma+2})}^3\right),\quad 0\leq t\leq T.\end{align}
Then   the conclusion of this lemma follows by the standard Bootstrap argument \cite{Taobook}.
\end{proof}

It is clear that the regularity of the solution of (\ref{model}) is determined by both the initial data and the potential. In many of the related physical applications, the initial input $u_0$ of (\ref{model}) is indeed rather smooth. So here we give an improved regularity result for the solution $u$ of (\ref{model}) when $u_0$ is from at least $H^{\gamma+4}(\bT)$. This, however, is neither necessary for establishing the convergence result in Theorem \ref{thm:nls} nor helps to make practical improvements on the performance of the proposed LRI scheme, which  will be illustrated by our numerical results. The reason is that the truncation error from the rough potential $\xi$ in the scheme will become the dominant term under the smoother data case, and this can be seen in the coming analysis of the local error.

\begin{lemma}\label{lm1p5} (A prior estimate in $H^{\gamma+4}$)
If  $\xi\in H^{\gamma+2}(\bT)$ and $u_0\in H^{\gamma+4}(\bT)$ for some $\gamma\geq0$ in (\ref{model}), then for some $T>0$, we have $u\in L^\infty((0,T),H^{\gamma+4}(\bT))$.
\end{lemma}
\begin{proof}
Similarly by Bootstrap, we assume
  $u\in L^\infty((0,T),H^{\gamma+4}(\bT))$ for some $T>0$. %By taking $H^{\gamma+2}$-norm on both sides of (\ref{lm1 eq0}) and by Sobolev's inequality,
%   \begin{equation}\label{lm1 eq2}\|u(t)\|_{H^\gamma}\leq \|u_0\|_{H^\gamma}+t\left[\|\xi\|_{H^{\gamma+2}}
%   \|u\|_{L^\infty((0,T),H^{\gamma+2})}+\lambda
%   \|u\|_{L^\infty((0,T),H^{\gamma+2})}^3\right],\quad 0\leq t\leq T.\end{equation}
By the equation (\ref{model}) and using Sobolev's inequality, we find
 for $0\leq t\leq T$,
\begin{equation}\label{lm1 eq3}\|\partial_tu(t)\|_{H^{\gamma+2}}\leq \|u\|_{L^\infty((0,T),H^{\gamma+4})}+\|\xi\|_{H^{\gamma+2}}
\|u\|_{L^\infty((0,T),H^{\gamma+4})}
  +\lambda\|u\|_{L^\infty((0,T),H^{\gamma+4})}^3.\end{equation}
 On the other hand, using integration-by-parts for (\ref{lm1 eq0}), with our notation $\partial_x^{-2}$ defined before, we have
    \begin{align}
    u(t)=&\fe^{it\partial_x^2}u_0+\partial_x^{-2}\left[
    \left(\xi+\lambda|u(t)|^2\right)u(t)-
    \fe^{it\partial_x^2}\left(\xi+\lambda|u_0|^2\right)u_0\right]
    -i\int_0^t\left[\widehat{(\xi u)}_0(s)
    +\lambda\widehat{(|u|^2u)}_0(s)\right]ds\nonumber\\
    &-\partial_x^{-2}\int_0^t\fe^{i(t-s)\partial_x^2}
    \left[\xi \partial_tu(s)+2\lambda\partial_tu(s) |u(s)|^2
    +\lambda \partial_t\overline{u}(s)u^2(s)\right]ds, \quad t\geq0.\label{lm1 eq1}
  \end{align}
By taking the $H^{\gamma+4}$-norm on both sides of (\ref{lm1 eq1}), we have
 for $0\leq t\leq T$,
   \begin{align*}
   \| u(t)\|_{H^{\gamma+4}}\leq &\|u_0\|_{H^{\gamma+4}}+\|
    \left(\xi+\lambda|u(t)|^2\right)u(t)\|_{H^{\gamma+2}}+
    \|\left(\xi+\lambda|u_0|^2\right)u_0\|_{H^{\gamma+2}}\nonumber\\
    &+\sqrt{|\bT|}\int_0^t\left[\left|\widehat{(\xi u)}_0
    \right|(s)
    +\lambda\left|\widehat{(|u|^2u)}_0\right|(s)\right]ds\\
    &+\int_0^t\left[
    \|\xi \partial_tu(s)\|_{H^{\gamma+2}}+2\lambda\|\partial_tu(s) |u(s)|^2\|_{H^{\gamma+2}}
    +\lambda \|\partial_t\overline{u}(s)u^2(s)\|_{H^{\gamma+2}}\right]ds.
    \end{align*}
 By further using Sobolev's inequality, we get
 \begin{align*}
 \|u(t)\|_{H^{\gamma+4}}
    \leq &\|u_0\|_{H^{\gamma+4}}+
    \|\xi\|_{H^{\gamma+2}}\|u_0\|_{H^{\gamma+4}}
    +\lambda\|u_0\|_{H^{\gamma+4}}^3+
    \|\xi\|_{H^{\gamma+2}}\|u(t)\|_{H^{\gamma+2}}
    +\lambda\|u(t)\|_{H^{\gamma+2}}^3\nonumber\\
    &+C\int_0^t\left[\|\xi\|_{H^{\gamma+2}}\|u(s)\|_{H^{\gamma+2}}
    +\lambda\|u(s)\|_{H^{\gamma+2}}^3\right]ds\\
    &+\int_0^t\left[
    \|\xi\|_{H^{\gamma+2}} +2\lambda \|u(s)\|_{H^{\gamma+2}}^2
    +\lambda\|u(s)\|_{H^{\gamma+2}}^2\right]
    \|\partial_tu(s)\|_{H^{\gamma+2}}ds,\quad 0\leq t\leq T,
    \end{align*}
    where $C>0$ depends only on the size of the torus.
  Then by plugging (\ref{lm1 eq11}) and (\ref{lm1 eq3}) into the above inequality, the conclusion of this lemma is obtained by Bootstrap argument.
\end{proof}

To avoid the technical difficulty for proof, we shall consider in the following $\gamma>\frac{1}{2}$ so that we have the algebraic property of the Sobolev space $H^\gamma$. The extension to $\gamma\geq0$ can be obtained by following the strategy in \cite{Lubich} to establish a weaker convergence rate of the scheme in order to obtain the boundedness of the numerical solution. Now we define numerical flow map of the LRI (\ref{vduhamel3}) as
\begin{align*}
  &\Psi_\tau(w)\\
  :=&\fe^{i\tau\partial_x^2}
 \Bigg[& \fe^{i\lambda\tau|w|^2}
  w+J_{1}(w)
  +J_{2}(w)-\lambda\tau^2
\xi|w|^2w-\frac{\tau^2}{2}
\xi^2w+\xi_1w
-2\xi_2\sin(-i\tau\partial_x)w-i\tau\widehat{\xi}_0w\Bigg],
\end{align*}
for some function $w$ on $\bT$, and then we have the following stability result.
\begin{lemma}\label{lm stable}(Stability)
  For $w_1,w_2\in H^{\gamma}(\bT)$ and $\xi\in
  H^{\gamma+2}(\bT)$ with some $\gamma>\frac12$, we have
  $$\|\Psi_\tau(w_1)-\Psi_\tau(w_2)\|_{H^\gamma}
  -\|w_1-w_2\|_{H^\gamma}\lesssim \tau
  \|w_1-w_2\|_{H^\gamma}.$$
\end{lemma}
\begin{proof}
  By directly taking the difference between $\Psi_\tau(w_1)$ and $\Psi_\tau(w_2)$ defined as above, and then taking the $H^\gamma$-norm on both sides, we find
  \begin{align}\label{lm2 eq-1}
 \|\Psi_\tau(w_1)-\Psi_\tau(w_2)\|_{H^\gamma}\leq&
 \zeta_1+\zeta_2,
  \end{align}
  where  we denote
 \begin{align*}
 \zeta_1=&\|\fe^{i\lambda\tau|w_1|^2}w_1
 -\fe^{i\lambda\tau|w_2|^2}w_2\|_{H^\gamma}
 +\|J_1(w_1)-J_1(w_2)\|_{H^\gamma}+\|J_2(w_1)-J_2(w_2)\|_{H^\gamma},\\
 \zeta_2=&\lambda\tau^2\|\xi(|w_1|^2w_1-|w_2|^2w_2)\|_{H^\gamma}
 +\frac{\tau^2}{2}
 \|\xi^2(w_1-w_2)\|_{H^\gamma}
 +\|\xi_1(w_1-w_2)\|_{H^\gamma}\\
 & +2\|\xi_2\sin(-i\tau\partial_x)(w_1-w_2)\|_{H^\gamma}
+\tau|\widehat{\xi}_0|\|w_1-w_2\|_{H^\gamma}.
 \end{align*}

For the estimate of $\zeta_1$, we refer the readers to \cite{kath1} (Lemma 3.1) for the result:
\begin{equation}\label{lm2 eq0}
\zeta_1- \|w_1-w_2\|_{H^\gamma}\lesssim\tau \|w_1-w_2\|_{H^\gamma}.
\end{equation}
Now we consider $\zeta_2$. Since $\xi,w_1,w_2\in H^{\gamma}$,
 so by Sobolev' inequality
we have
\begin{align}\label{lm2 eq1}
\zeta_2\lesssim \tau \|w_1-w_2\|_{H^\gamma}+\|\xi_1(w_1-w_2)\|_{H^\gamma}
+\|\xi_2\sin(-i\tau\partial_x)(w_1-w_2)\|_{H^\gamma}.
\end{align}
For the two regularized potentials $\xi_1$ and $\xi_2$ in (\ref{xi12}), by their definitions in Fourier space and Taylor's expansion, we have
$$\xi_1=\sum_{l\neq0}\fe^{i\mu_l(x+L)}
\frac{1-\fe^{i\tau\mu_l^2}}{\mu_l^2}\widehat{\xi}_l
=-i\tau\sum_{l\neq0}\fe^{i\mu_l(x+L)}\int_0^1\fe^{i\mu_l^2\tau
\theta}d\theta\widehat{\xi}_l,
$$
and
$$\xi_2=\sum_{l\neq0}\fe^{i\mu_l(x+L)}
\frac{i-i\fe^{i\tau\mu_l^2}-\tau\mu_l^2\fe^{i\tau\mu_l^2}}{i\tau\mu_l^3}
\widehat{\xi}_l
=\tau\sum_{l\neq0}\fe^{i\mu_l(x+L)}\mu_l\left[\int_0^1(1-\theta)
\fe^{i\tau\mu_l^2\theta}d\theta-\int_0^1
\fe^{i\tau\mu_l^2\theta}d\theta\right]\widehat{\xi}_l.
$$
Therefore, by Parseval's identity, we have
\begin{equation}\label{xi12 est}
\|\xi_1\|_{H^\gamma}\lesssim\tau \|\xi\|_{H^{\gamma}},\quad  \|\xi_2\|_{H^\gamma}\lesssim\tau \|\xi\|_{H^{\gamma+1}},
\end{equation}
and then we get from (\ref{lm2 eq1}),
\begin{align}\label{lm2 eq2}
  \zeta_2\lesssim \tau \|w_1-w_2\|_{H^\gamma}+
  \tau\|\xi\|_{H^\gamma}\|w_1-w_2\|_{H^\gamma}+
  \tau\|\xi\|_{H^{\gamma+1}}\|w_1-w_2\|_{H^\gamma}\lesssim\tau \|w_1-w_2\|_{H^\gamma} .
\end{align}
By plugging (\ref{lm2 eq2}) and (\ref{lm2 eq0}) into (\ref{lm2 eq-1}) gives the result of the lemma.
\end{proof}

Next, for the local truncation error of the LRI scheme, we have the following estimate.

\begin{lemma}\label{lm local}(Local error) Under the assumption in Theorem \ref{thm:nls}, we have the following estimate of the local truncation error of the scheme (\ref{vduhamel3}),
$$\|u(t_{n+1})-\Psi_\tau(u(t_n))\|_{H^\gamma}\lesssim \tau^3,\quad 0\leq n<\frac{T}{\tau}. $$
\end{lemma}
\begin{proof}
Noting by our notation (\ref{notation Vn}) that $V_n(\tau)=\fe^{-i\tau\partial_x^2}u(t_{n+1})$ and $V_n=u(t_n)$, we have
\begin{align}\label{lm3 eq0}
\|u(t_{n+1})-\Psi_\tau(u(t_n))\|_{H^\gamma}=  \|V_n(\tau)-K(V_n)-P(V_n)\|_{H^\gamma},
\end{align}
where we denote $\Psi_\tau(w)=\fe^{i\tau\partial_x^2}
(K(w)+P(w))$ with $$K(w):=\fe^{i\lambda\tau|w|^2}w+J_1(w)+J_2(w),$$ and
\begin{align}\label{P tilde def}
 P(w):=\tilde{P}(w)-\lambda\tau^2
\xi|w|^2w-\frac{\tau^2}{2}
\xi^2w,\qquad \tilde{P}(w):=\xi_1w-2\xi_2\sin(-i\tau\partial_x)w
-i\tau\widehat{\xi}_0w,
\end{align}
for some $w$ on $\bT$. Based on the assumption, we have by Lemma \ref{lm1}, $u(t)\in H^{\gamma+2}$ for $0\leq t\leq T$, and so does the filtered variable $v(t)$ in (\ref{v def}) or $V_n(\rho)$ in (\ref{notation Vn}) for $0\leq \rho\leq\tau$ and $0\leq n<T/\tau$. Then by (\ref{vduhamel1}), for the approximation used in (\ref{Vn 1}), we have $\|V_n(\rho)-V_n\|_{H^\gamma}\lesssim\tau$. Then by Taylor's expansion and noting that $\xi\in H^{\gamma+2}$, as has also been explained in the derivation of (\ref{Vn 1}) and (\ref{vduhamel2}), we have
$$\| V_n(\tau)-
  V_n-A_1(V_n)-A_2(V_n)-\mathcal{N}(V_n)\|_{H^\gamma}\lesssim \tau^3,\quad
  0\leq n<T/\tau,$$
  with $\mathcal{N}$ and $A_1,\,A_2$ defined respectively in (\ref{Nn def}) and (\ref{A12 def}).
So by triangle inequality, (\ref{lm3 eq0}) gives
  \begin{align}\label{lm3 eq1}
\|V_n(\tau)-K(V_n)-P(V_n)\|_{H^\gamma}\lesssim&\tau^3+
\|V_n+\mathcal{N}(V_n)-K(V_n)\|_{H^\gamma}
+\left\|A_2(V_n)-\tilde{P}(V_n)\right\|_{H^\gamma}\nonumber\\
&+
\left\|A_1(V_n)+\lambda\tau^2\xi|V_n|^2V_n+\frac{\tau^2}{2}
 \xi^2V_n\right\|_{H^\gamma}.
  \end{align}
On the right-hand-side of (\ref{lm3 eq1}), the second term
is the truncation error of integrating the equation (\ref{model}) apart from
 the potential term, which reads the same as has been considered in \cite{kath1}. Noting by our \emph{a prior} estimate result in Lemma \ref{lm1} that the regularity on the solution $u$ of (\ref{model}) satisfies the assumptions in \cite{kath1} (Lemma 2.7), i.e. $u\in H^{\gamma+2}$, so here we directly quote the result:
\begin{equation}\label{lm3 eq4}
\|V_n+\mathcal{N}(V_n)-K(V_n)\|_{H^\gamma}\lesssim\tau^3. \end{equation}
We now focus on the last two terms in (\ref{lm3 eq1}) related to the approximations of the potential term. Firstly, for the approximation of $A_1(V_n)$ as in (\ref{A1 app}), we note that all the truncations here are from $\fe^{\pm is\partial_s}\approx1$. So similarly as before by Taylor's expansion,  we can clearly have
\begin{equation}\label{lm3 eq3}\left\|A_1(V_n)+\lambda\tau^2\xi|V_n|^2V_n+\frac{\tau^2}{2}
 \xi^2V_n\right\|_{H^\gamma}\lesssim\tau^3.\end{equation}
  Next, we estimate the error in approximating $A_2(V_n)$: $\|A_2(V_n)-\tilde{P}(V_n)\|_{L^2}$. Using Taylor's expansion and our estimate for $\xi_2$ in (\ref{xi12 est}), we see
 $$\|\xi_2\sin(-i\tau\partial_x)V_n+\xi_2i\tau\partial_xV_n\|_{H^\gamma}
 \lesssim\tau^2\|\xi_2\|_{H^\gamma}\|V_n\|_{H^{\gamma+2}}\lesssim
 \tau^3\|\xi\|_{H^{\gamma+1}}\|V_n\|_{H^{\gamma+2}}\lesssim\tau^3,$$
and so by noting $\tilde{P}$ defined in (\ref{P tilde def}) and using the triangle inequality, we get
 \begin{align}\label{lm3 eq2}
 \left\|A_2(V_n)-\tilde{P}(V_n)\right\|_{L^2}\lesssim
 \left\|A_2(V_n)-\xi_1V_n-2\xi_2i\tau\partial_xV_n
 +i\tau\widehat{\xi}_0V_n\right\|_{H^\gamma}
 +\tau^3.
 \end{align}
 By noting (\ref{A2 def}), and the exact integration in (\ref{tofilter}), we have
 \begin{align*}
&A_2(V_n)-\xi_1V_n-2\xi_2i\tau\partial_xV_n
 +i\tau\widehat{\xi}_0V_n\\
 =&-i\sum_{l_1+l_2=l}\fe^{i\mu_l(x+L)}\int_0^\tau
 \fe^{is\mu_{l_1}^2}\left(\fe^{2is\mu_{l_1}\mu_{l_2}}
 -1-2i\mu_{l_1}\mu_{l_2}s\right)ds\widehat{\xi}_{l_1}
 \widehat{(V_n)}_{l_2},
 \end{align*}
 and so by Taylor's expansion and Parseval's identity,
 \begin{align*}
\left\|A_2(V_n)-\xi_1V_n-2\xi_2i\tau\partial_xV_n
 +i\tau\widehat{\xi}_0V_n\right\|_{H^\gamma}
 \lesssim&\tau^3\|\xi\|_{H^{\gamma+2}}\|V_n\|_{H^{\gamma+2}}\lesssim\tau^3.
 \end{align*}
 Therefore, (\ref{lm3 eq2}) gives
$$ \left\|A_2(V_n)-\tilde{P}(V_n)\right\|_{L^2}\lesssim\tau^3,$$
and in combine with (\ref{lm3 eq4}) and (\ref{lm3 eq3}), we further get from (\ref{lm3 eq1})
 $$\|V_n(\tau)-K(V_n)-P(V_n)\|_{H^\gamma}\lesssim\tau^3.$$
  \end{proof}

 As we mentioned before, here it can be seen from the above lemma that as long as $u\in H^{\gamma+2}$, the leading order term in the truncation error of the LRI scheme is determined by the regularity of $\xi$.

 Now combing the above lemmas for local error and stability, we can give the proof of our convergence theorem.

  \emph{Proof for Theorem \ref{thm:nls}}:
  \begin{proof}
   By the triangle inequality and the scheme (\ref{vduhamel3}), we find
    for $0\leq n<T/\tau$,
    \begin{align*}
   \|u(t_{n+1})-u^{n+1}\|_{H^\gamma}\leq
   \|u(t_{n+1})-\Psi_\tau(u(t_n))\|_{H^\gamma}+
   \|\Psi_\tau(u(t_n))-\Psi_\tau(u^n)\|_{H^\gamma}.
    \end{align*}
 Under an induction argument for the boundedness of the numerical solution
 $u^n$, and then using Lemma \ref{lm local} and Lemma \ref{lm stable}, we get
 \begin{align*}
   \|u(t_{n+1})-u^{n+1}\|_{H^\gamma}-\|u(t_{n})-u^{n}\|_{H^\gamma}\lesssim\tau^3
   +
   \tau \|u(t_{n})-u^{n}\|_{H^\gamma}.
    \end{align*}
Then the theorem is proved by the Gronwall's inequality when $\tau\leq\tau_0$ for some constant $\tau_0>0$ depends on $T$ and the norms of $u_0,\xi$.
  \end{proof}
\subsection{Random potential}\label{sec:random} To finish this section, we discuss about the random potential that will be considered for our numerical experiments later.  By following the strategy from the literature  \cite{sde3,debussche2,CAL2},
we shall construct the random potential $\xi(x)$ in (\ref{model}) for our numerical tests as:
\begin{equation}\label{xi0}\xi(x)=\frac{1}{C_0}\sum_{l=-N/2}^{N/2-1}
\frac{1}{\nu_l^\theta}\widehat{\xi}_l\fe^{i\mu_l(x+L)}+c.c.,\quad x\in\bT,\quad
\mbox{with}\quad
\nu_l=\left\{\begin{split}
 &l,\quad\, \mbox{if}\ l\neq0,\\
  &1,\quad\mbox{if}\ l=0,
  \end{split}\right.
\end{equation}
where $N>0$ is the number of spatial grids, $\nu_l$ is a regularizing factor with $\theta\geq0$, $C_0>0$ is some scaling factor, the sequence $\widehat{\xi}_l$ are i.i.d. complex random variables, and c.c. denotes the complex conjugate. This  setup (\ref{xi0}) is consistent with the periodic boundary condition in (\ref{model}) for computations. By tuning the parameters, we are able to get $\xi(x)$ with the desired regularity on $\bT$, in order to verify our theoretical convergence result.  For the random Fourier components $\widehat{\xi}_l$, we will consider the case of the uniform distribution as has been considered in \cite{CAL2}. Such potential can also be viewed as a continuous analogy of the on-site random potential for the discrete model (\ref{discrete}) in the physical work \cite{prl2,prl1,pre1}. Moreover, we will also consider the normal distribution for $\widehat{\xi}_l$ in our numerical examples.
As has been shown in \cite{whitenoise}, for $\theta=0,\ C_0=O(1)$ and $N\to\infty$, when $\widehat{\xi}_l$ are taken as complex-valued standard normal distributions, the $\xi(x)$ given in (\ref{xi0}) is a Gaussian white noise which has been considered for (\ref{model}) in the physical and mathematical studies \cite{Conti,debussche3,debussche1,pre3}.

In details, we shall consider the scaling factor $C_0=O(N)$ (same as in  discrete inverse Fourier transform) in (\ref{xi0}). This scaling factor here is also served  to control the peak value of $\xi$, in order to avoid numerical stability issue, otherwise the constant $\tau_0>0$ in Theorem \ref{thm:nls} may be very small.
To verify our theoretical results in Section \ref{sec:thm}, we shall  take $\theta=2$ and $\xi_l$ as the uniform distribution in a bounded interval, so that the potential $\xi(x)\in H^2(\bT)$ as $N\to\infty$. Meanwhile, as explained above, we will test the case that $\widehat{\xi}_l$ satisfy the normal distribution. As will be shown by our numerical results, for both cases of the considered random potentials with initial data $u_0\in H^{2}$, the expectation of the global error of the fully discretized LRI scheme (\ref{vduhamel3}) (with Fourier pseudo-spectral method in space) reads
$$\mathbb{E}\left(\|u-u^n\|_{L^2}\right)=O(\tau^2)+O(h^2),\quad 0\leq n\leq T/\tau,$$
for a fixed time $T>0$. The temporal error part reads the same as has been stated in the Theorem \ref{thm:nls}. For the spatial error part in the above, we can formally see it as follows. In evaluation of the LRI scheme (\ref{vduhamel3}), we need to compute the terms
$\fe^{i\tau\partial_x^2}(
\xi|u^n|^2u^n)$, $\fe^{i\tau\partial_x^2}(
\xi^2u^n)$ and $\fe^{i\tau\partial_x^2}(\widehat{\xi}_0u^n)$. Therefore, if $\xi\in H^{2}$ and $u\in H^2$, the interpolation error from them is then known to be  $O(h^2)$ under the $L^2$-norm \cite{Shen}. With smoother initial data, e.g. for $u_0$ better than $H^2$ so that $u$ is smoother as stated in Lemma \ref{lm1p5}, the spatial error will not be improved,  because the dominated error will then be the part from the interpolation of $\xi$, which is still $O(h^2)$.

Afterwards, we will perform the corresponding numerical tests under the less regular situation for potential $\xi(x)$, e.g. $\theta=0$ in (\ref{xi0}).
In these random potential cases, as we shall show by extensive numerical studies in the next section, for the fully discretized LRI scheme (\ref{vduhamel3}),
the expectation of the global error appears to be
$$\mathbb{E}\left(\|u-u^n\|_{L^2}\right)=O(\tau^p)+O(h^q),\quad 0\leq n\leq T/\tau,$$
with some $1<p<2$ and $0<q<1$ for a fixed $T>0$.

%\section{Extension to disordered nonlinear Dirac equation}

\section{Numerical results}\label{sec:result}
In this section, we are going to conduct numerical tests of the presented methods, i.e. the classical methods from Section \ref{sec:method} and the proposed LRI method (\ref{vduhamel3}). We shall present the numerical results to illustrate and compare the accuracy of the methods.

\subsection{Test on discrete D-NLS}
Firstly, let us present the test of the classical methods, i.e. the splitting methods and the finite difference methods, on the discrete model (\ref{discrete}) truncated to a finite-size periodic lattice:
\begin{equation}\left\{\begin{split}
&i\dot{u}_l(t)=-J[u_{l+1}(t)+u_{l-1}(t)]
+\xi_lu_l(t)+\lambda|u_l(t)|^2u_l(t),\quad t>0,\ -N\leq l\leq N,\\
&u_{-N}(t)=u_{N}(t),\quad u_{-N-1}(t)=u_{N-1}(t),\quad
u_{N+1}(t)=u_{-N+1}(t),\quad t\geq0,\label{model2}\\
&u_l(0)=u_l^0,\quad  -N\leq l\leq N, \end{split}\right.
\end{equation}
 as have been done in the physical studies \cite{review,prl2,prl1,pre1}. This test is made here to confirm the optimal convergence rate of the classical methods on the discrete model, and to compare later with their performances on the continuous D-NLS (\ref{model}). In (\ref{model2}), we take $J=1$, $N=128$ for the lattice size and $\lambda=1$ as for the defocusing interaction. The discrete random potential $\xi_l$ are taken as i.i.d. uniform distribution in $[-1,1]$ or standard normal distribution, and the initial data is imposed as zero except on three central sites:
$$u_{-1}=u_{1}=1/4,\quad u_0=1/2.$$
The detailed schemes of the Strang splitting method and the semi-implicit finite difference (FD) method on the discrete model (\ref{model2}) are given in Appendix \ref{sec:app}. The expectation of their relative errors $\mathbb{E}(\|u-u^n\|_{l^\infty}/\|u\|_{l^\infty})$ at $t_n=T=1$ under 100 samples are plotted in Figure \ref{fig:discret}. Here the reference solution is obtained for each sample by Strang splitting method with time step $\tau=10^{-4}$.

\begin{figure}[t!]
$$\begin{array}{cc}
\psfig{figure=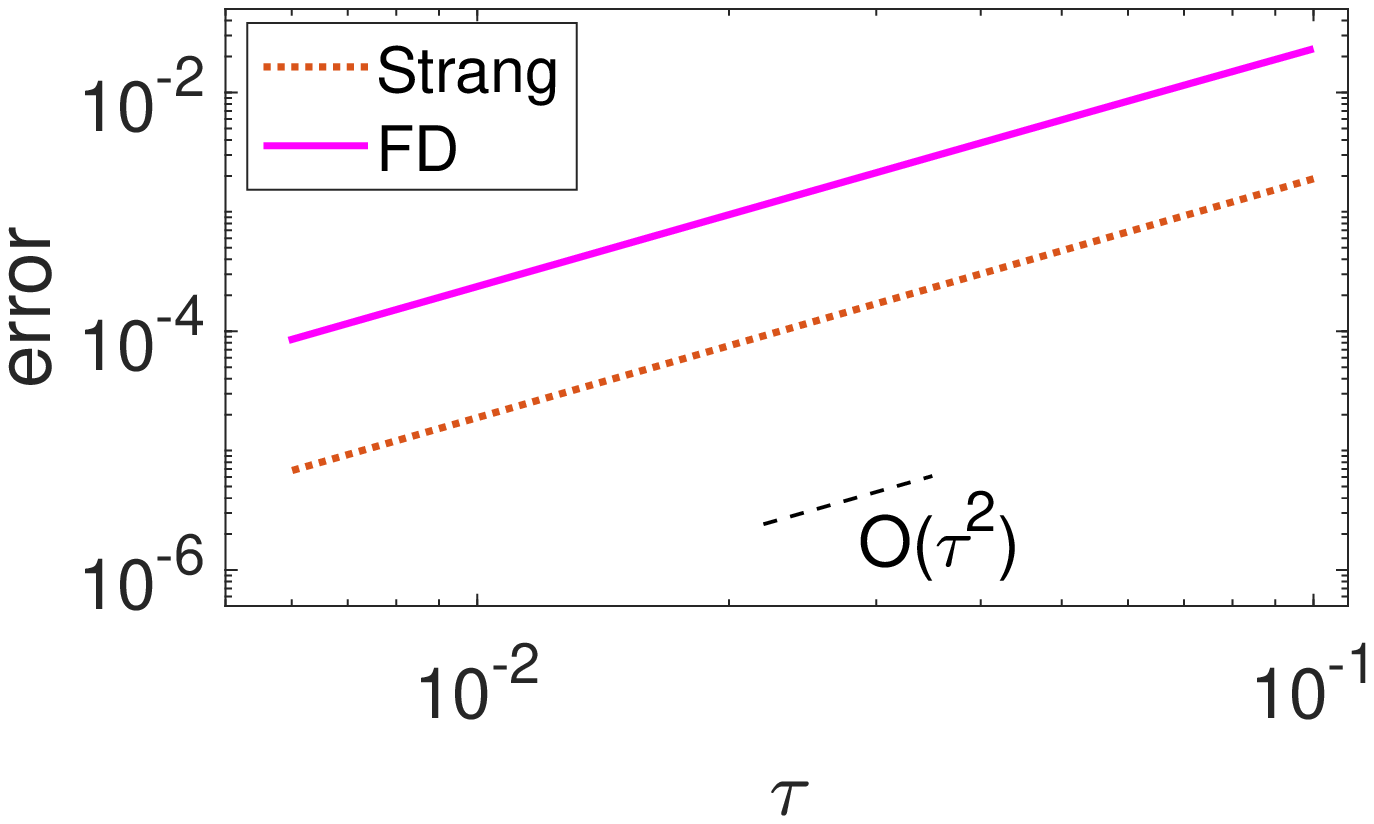,height=4.5cm,width=7cm}&
\psfig{figure=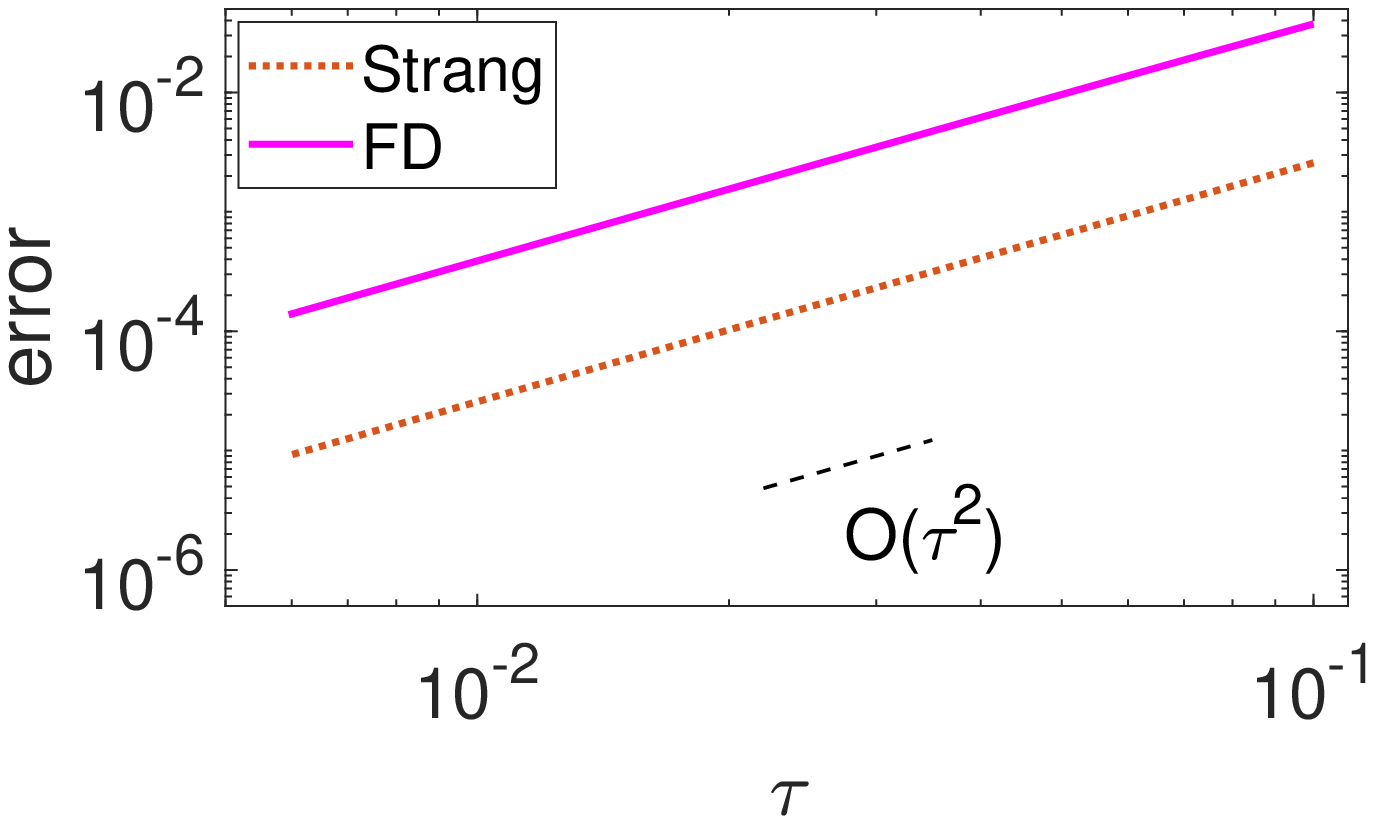,height=4.5cm,width=7cm}
\end{array}$$
\caption{Convergence of the classical schemes on (\ref{model2}): the error $\mathbb{E}(\|u-u^n\|_{l^\infty}/\|u\|_{l^\infty})$ with potential generated by  uniform distribution (left) or normal distribution (right).}
\label{fig:discret}
\end{figure}

It can seen from Figure \ref{fig:discret} that the errors of the two classical integrators converge at the second order rate which is their optimal convergence rate, and the splitting scheme is more accurate than the finite difference scheme. They and their higher order extensions can simulate the discrete D-NLS (\ref{discrete}) accurately and efficiently.

\begin{figure}[t!]
$$\begin{array}{cc}
\psfig{figure=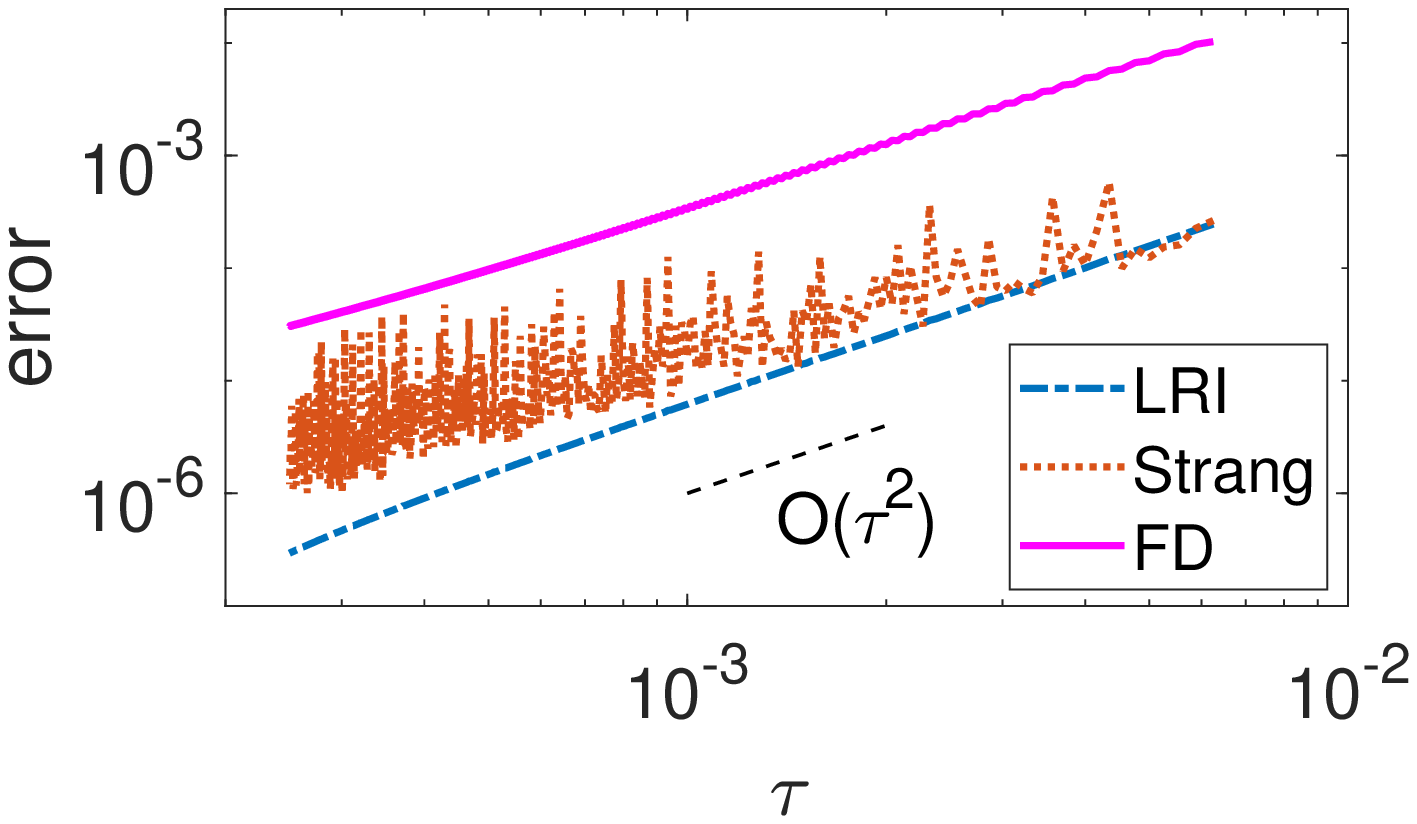,height=4.5cm,width=7cm}&
\psfig{figure=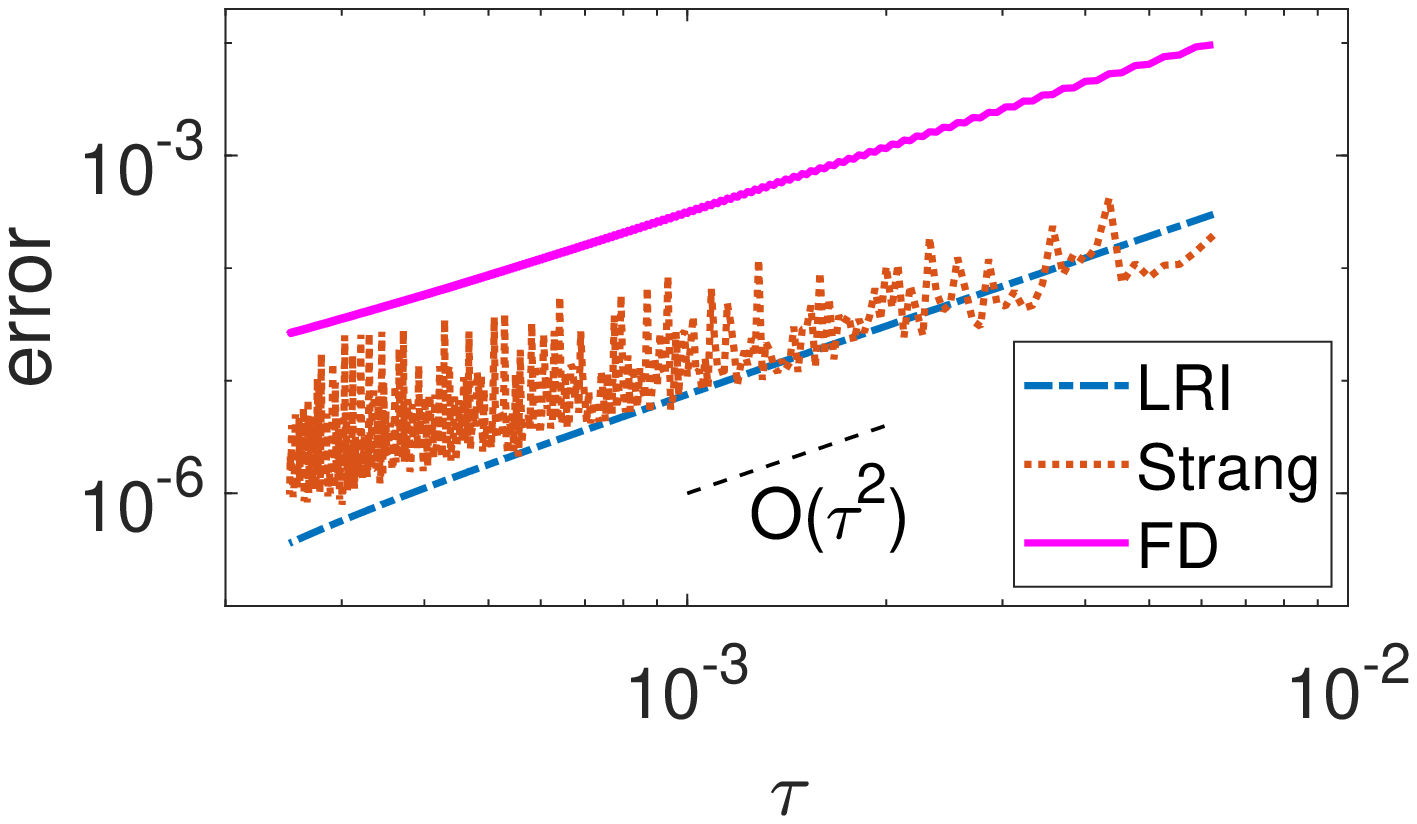,height=4.5cm,width=7cm}
\end{array}$$
\caption{Temporal convergence in Example \ref{ex1} with $\xi\in H^2$ and smooth $u_0$: the error $\mathbb{E}(\|u-u^n\|_{L^2}/\|u\|_{L^2})$  under potential generated by uniform distribution (left) and normal distribution (right).}
\label{fig:V2smooth}
\end{figure}

\begin{figure}[t!]
$$\begin{array}{cc}
\psfig{figure=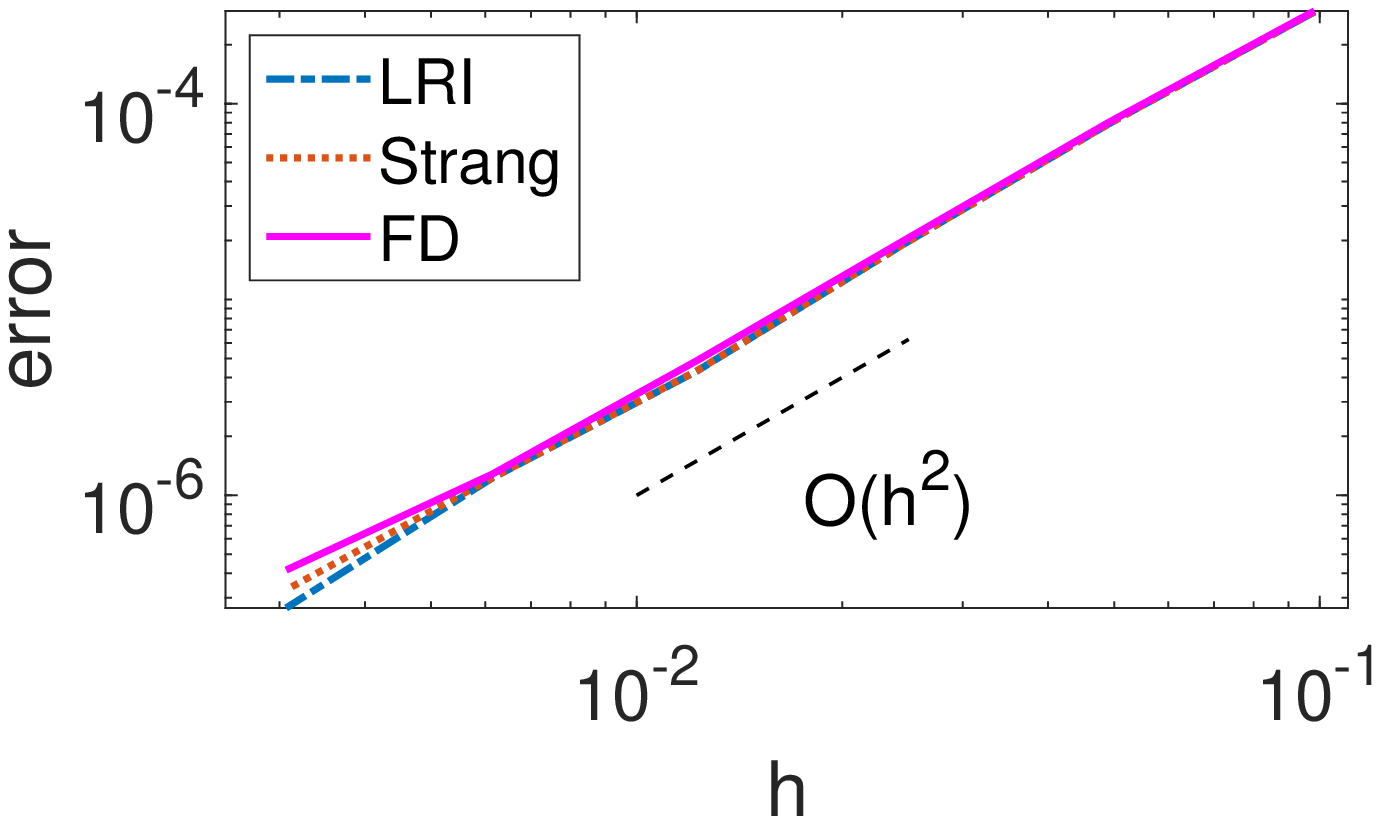,height=4.5cm,width=7cm}&
\psfig{figure=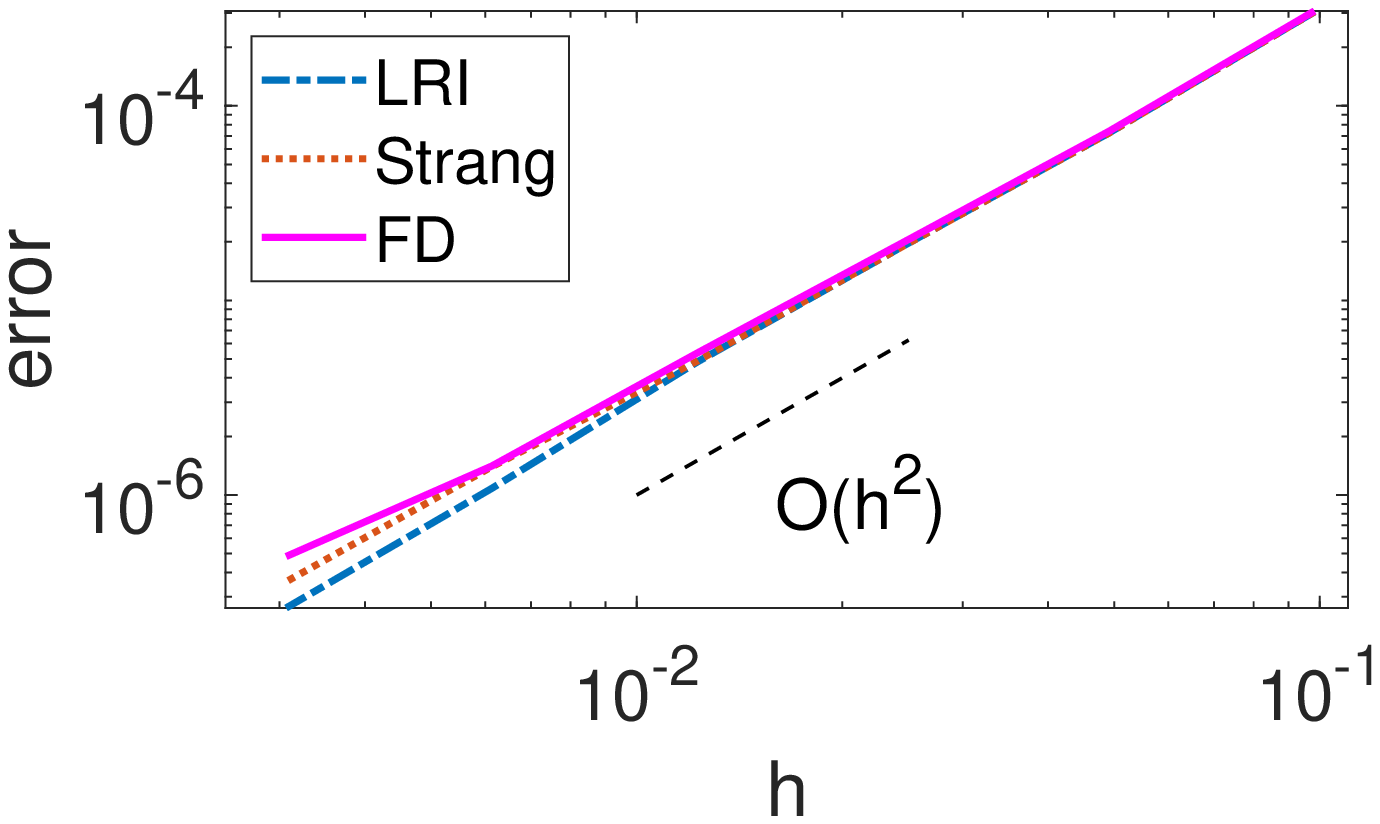,height=4.5cm,width=7cm}
\end{array}$$
\caption{Spatial convergence in Example \ref{ex1} with smooth $u_0$: the error $\mathbb{E}(\|u-u^n\|_{L^2}/\|u\|_{L^2})$ under potential generated by uniform distribution (left) and normal distribution (right).}
\label{fig:V2smoothspace}
\end{figure}

\begin{figure}[t!]
$$\begin{array}{ccc}
\psfig{figure=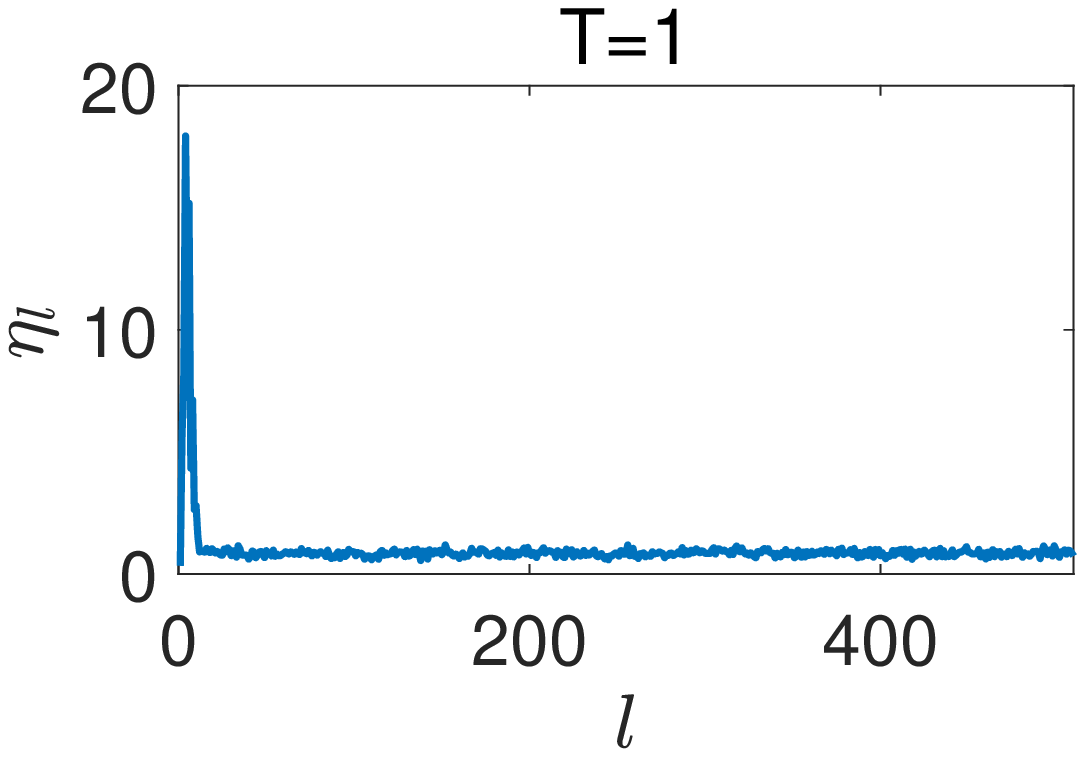,height=4.0cm,width=4.5cm}&
\psfig{figure=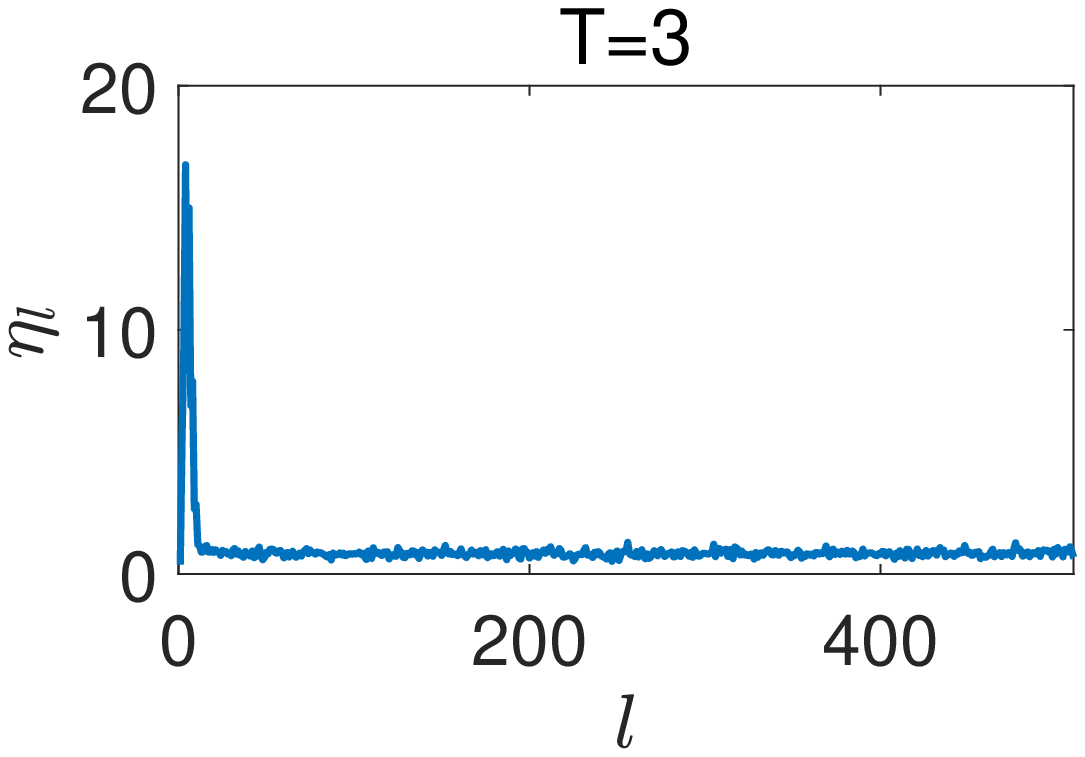,height=4.0cm,width=4.5cm}&
\psfig{figure=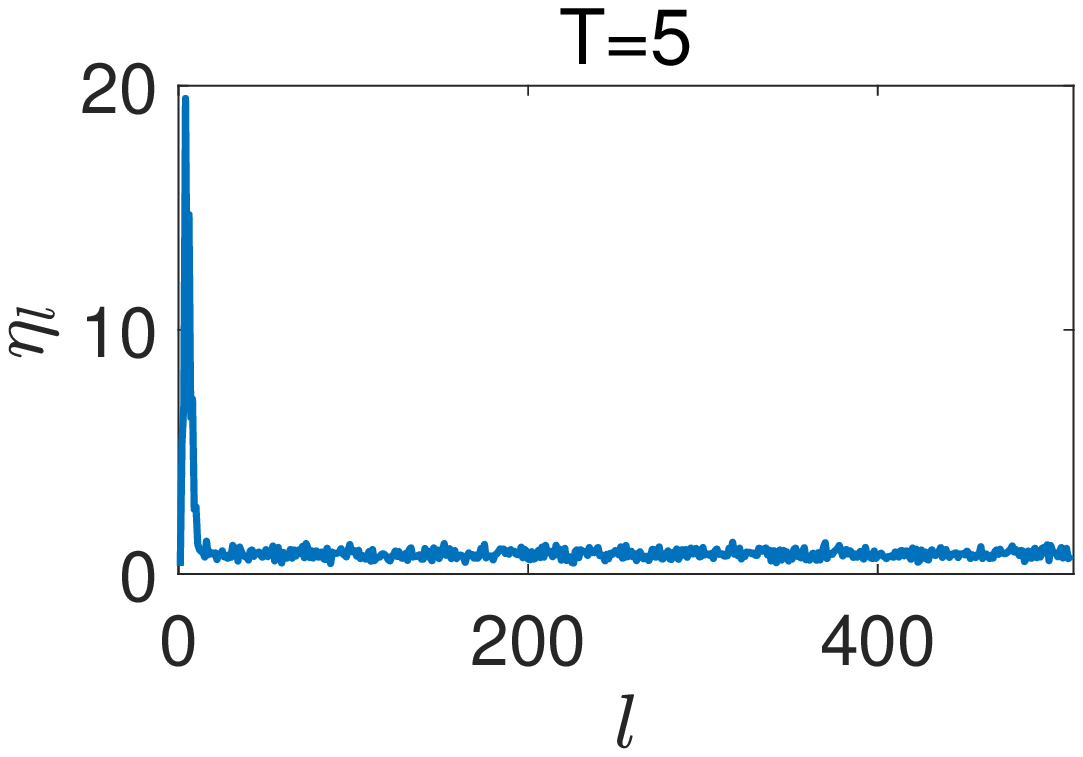,height=4.0cm,width=4.5cm}
\end{array}$$
\caption{Profile of scaled Fourier models $\eta_l(t)=\mathbb{E}(|\widehat{u}_l(t)|l^4)$ in Example \ref{ex1} with $u_0$ smooth at different time $t=T$ (potential generated by normal distribution).}
\label{fig:smooth mode}
\end{figure}

\begin{figure}[t!]
$$\begin{array}{cc}
\psfig{figure=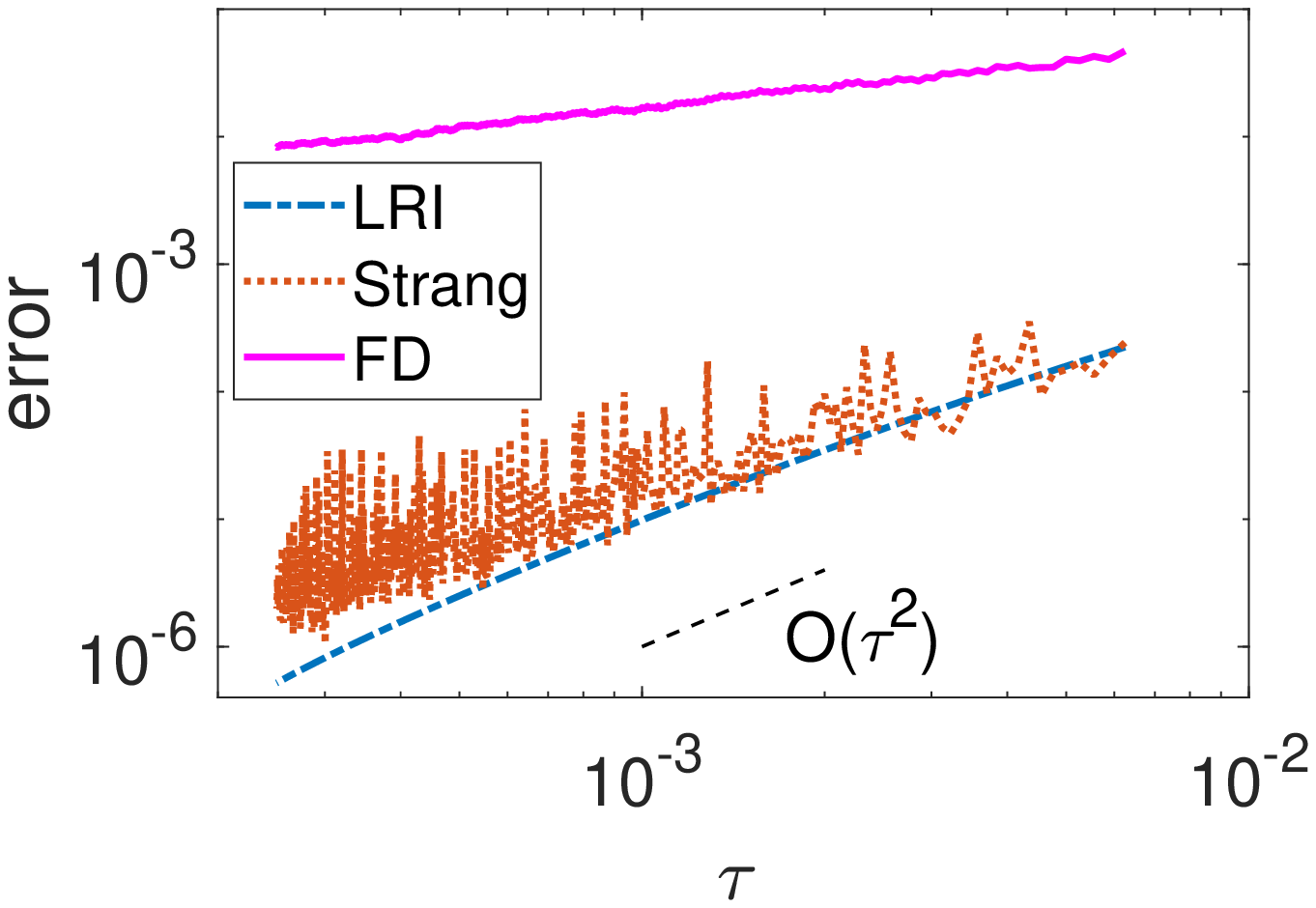,height=4.5cm,width=7cm}&
\psfig{figure=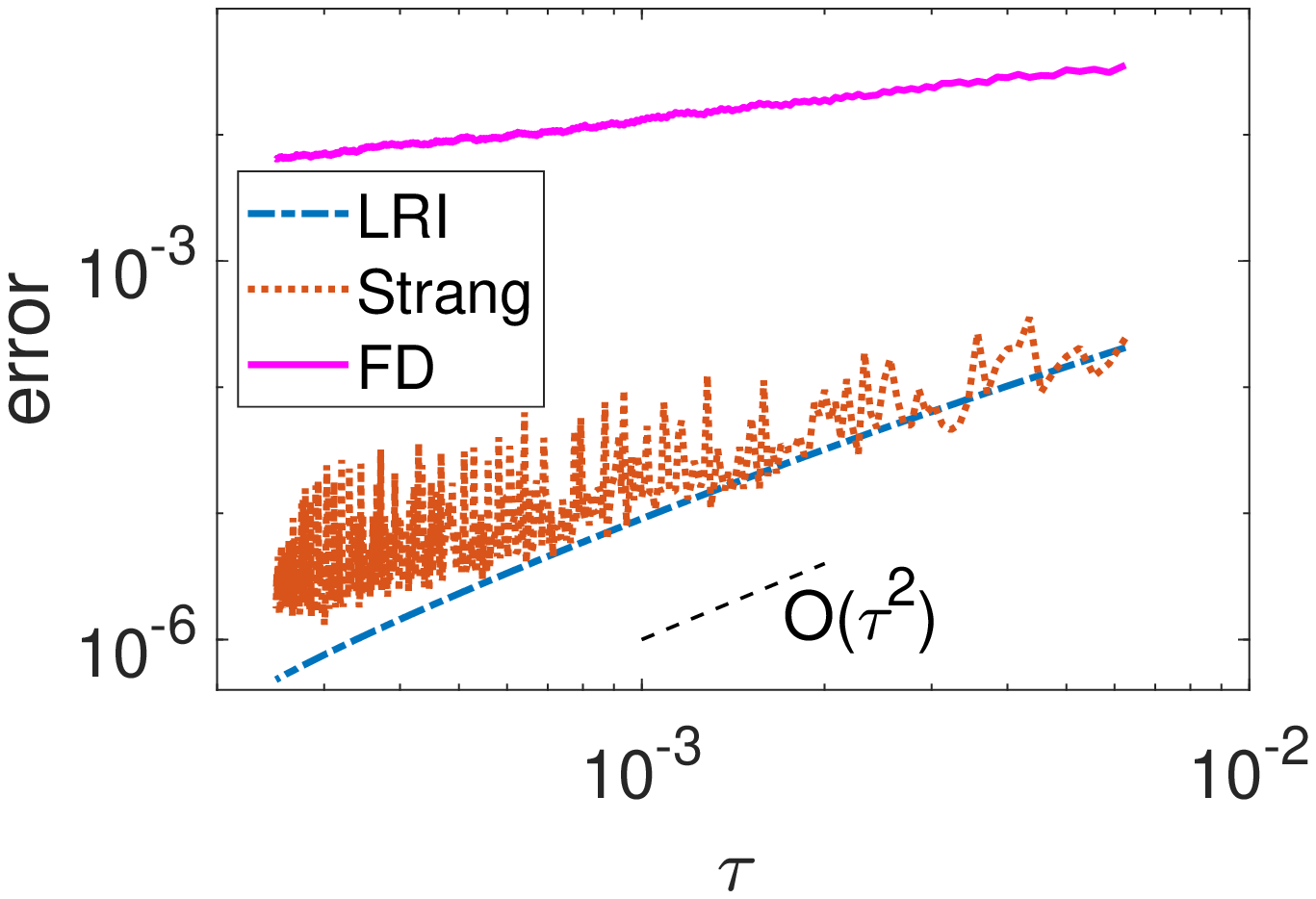,height=4.5cm,width=7cm}
\end{array}$$
\caption{Temporal convergence in Example \ref{ex1} with $u_0\in H^2$: the error $\mathbb{E}(\|u-u^n\|_{L^2}/\|u\|_{L^2})$  under potential generated by uniform distribution (left) or normal distribution (right).}
\label{fig:V2H2}
\end{figure}

%\begin{figure}[t!]
%$$\begin{array}{ccc}
%\psfig{figure=T1normal.eps,height=4.5cm,width=4.9cm}&
%\psfig{figure=T3normal.eps,height=4.5cm,width=4.9cm}&
%\psfig{figure=T5normal.eps,height=4.5cm,width=4.9cm}
%\end{array}$$
%\caption{Profile of the scaled Fourier models $\eta_l(t)=\mathbb{E}(|\widehat{u}_l(t)|l^2)$ at different time $t=T$ under potential generated by normal distribution.}
%\label{fig:mode}
%\end{figure}

\subsection{Test on continuous D-NLS}
Now, we conduct tests of the classical methods: the Strang splitting method (\ref{strang}) and the FD method (\ref{FD}), and the proposed LRI method (\ref{vduhamel3}) on the continuous D-NLS (\ref{model}). We will consider three numerical examples as follows.

\begin{example}\label{ex1} ($H^2$-potential)
Firstly, to verify our theoretical result, we consider a numerical experiment of a defocusing case $\lambda=1$ on a torus $\bT=(-\pi,\pi)$ with $\xi(x)\in H^2(\bT)$ for (\ref{model}). We use the strategy as mentioned before in (\ref{xi0}) to construct the spatial random potential $\xi$. Precisely, we take% the  $C_0=N$ and
%$
%\xi_l=\omega_{1,l}+i\omega_{2,l},$
\begin{equation}\label{xi}
\xi_0(x)=|\partial_{x,N}|^{-\theta}(\omega_1+i\omega_2)+c.c.,\quad
\xi(x)=\frac{2\xi_0(x)}{\|\xi_0\|_{L^\infty}},\quad x\in\bT,
\end{equation}
where $\omega_1,\,\omega_2\in\bR^N$ are two i.i.d. random vectors of length $N>0$ with $N$ equal to the number of spatial grids, and the pseudo-differential operator $|\partial_{x,N}|^{-\theta}$ reads: for Fourier modes $l=-N/2,\ldots, N/2-1$,
\begin{equation*}
 \left(|\partial_{x,N}|^{-\theta}\right)
 _l:=\left\{\begin{split}
 &|l|^{-\theta},\quad \mbox{if}\ l\neq0,\\
  &1,\qquad\ \, \mbox{if}\ l=0,
  \end{split}\right.\qquad \theta\geq0.
\end{equation*}
%where $\omega_{1,l},\,\omega_{2,l}\in\bR$ are i.i.d. random variables for $-N/2\leq l\leq N/2-1$ with $N$ equal to the number of spatial grids, to generate the random potential $\xi_0$, and then we further scale it: $\xi=2\xi_0/\|\xi_0\|_{L^\infty}$.
 We fix $\theta=2$ in this example for $\xi(x)\in H^2(\bT)$. The random variables  $\omega_{1},\,\omega_{2}$ are taken either as the uniform distribution in $[-1,1]$ or as the standard normal distribution.  For the initial input in (\ref{model}), we first take a smooth periodic initial function:
$$u_0(x)=\frac{\cos(x)+i\sin(2x)}{1+\sin(x)^2},\quad x\in\bT.$$

We use 100 samples of the potential from (\ref{xi}), and compute the expectation of the $L^2$-norm of the relative  error: \begin{equation}\label{error def}
error=\mathbb{E}(\|u(\cdot,t=t_n)-u^n\|_{L^2}/\|u(\cdot,t=t_n)\|_{L^2}),
\end{equation}
at $t_n=T=0.5$. The reference solution of each sample is obtained by Strang splitting method under very small mesh size, e.g. $\tau=10^{-5}$ and $N=2^{13}$. To test the temporal accuracy, we fix spatial mesh size $h=2\pi/N$ with $N=2^{13}$. The errors of the Strang splitting method (\ref{strang}), the FD method (\ref{FD}) and the LRI method (\ref{vduhamel3}) are plotted in Figure \ref{fig:V2smooth}. The spatial errors of the methods are illustrated in Figure \ref{fig:V2smoothspace} under fixed time step $\tau=10^{-5}$.

To illustrate the regularity of the solution under the setup, we apply the LRI method to investigate the asymptotical decay rate of the Fourier modes $\widehat{u}_l(t)$
of (\ref{model}) with respect to $l$. We consider the potential (\ref{xi}) generated by the standard normal distribution and compute the average of the scaled modes  $\eta_l(t):=\mathbb{E}(|\widehat{u}_l(t)|l^4)$. In Figure \ref{fig:smooth mode}, we plot $\eta_l(t)$  for modes $0<l<N/2$ at different time $t$ obtained by the LRI (\ref{vduhamel3}) under $N=2^{10}$ and $\tau=5\times10^{-5}$. The corresponding result for $\xi$ generated by the uniform distribution is very similar, and so it is omitted here for brevity.

To further test the impact from the smoothness of the initial data and verify the convergence of our LRI scheme under the precise setup as stated in Theorem \ref{thm:nls}, we now switch to a $H^2$-data constructed similarly as the potential:
$$u_0(x)=\frac{|\partial_{x,N}|^{-2}\mathcal{U}_N}{
\||\partial_{x,N}|^{-2}\mathcal{U_N}\|_{L^\infty}},\quad x\in \bT,
\quad \mathcal{U}_N=\mathrm{rand}(N,1)+i\,\mathrm{rand}(N,1),$$
with $\mathrm{rand}(N,1)$ uniformly distributed in $[0,1]^N$. Under the same computational parameters, the corresponding temporal error is plotted in Figure \ref{fig:V2H2}.
For the spatial error, the results from the Strang splitting method and the LRI method in this rough initial data case are very similar as in the smooth data case shown in Figure \ref{fig:V2smoothspace}. The FD method in this case significantly loses convergence rate, which is as expected. Since the spatial discretization error under rough data case is not what we aim to emphasize, so the results are omitted here for brevity.
%We take and the mesh size $h=2\pi/N$.
\end{example}

Based on the numerical results from Figures \ref{fig:V2smooth}-\ref{fig:V2H2}, we can see that

1) The convergence results for $\xi$ generated from the uniform distribution and from the normal distribution in (\ref{xi}) are very similar (see Figures \ref{fig:V2smooth}\&\ref{fig:V2smoothspace}\&\ref{fig:V2H2}). For the uniform distribution case of potential $\xi\in H^2$ and the rough initial data $u_0\in
H^2$, the presented LRI scheme (\ref{vduhamel3}) can reach its optimal second order
convergence rate in time (see Figure \ref{fig:V2H2}), which justifies our theoretical convergence result in Theorem
\ref{thm:nls}. For normal distribution case, the convergence of LRI is observed in the sense of expectation, and such strong convergence needs further rigorous analysis.

2) Under the potential $\xi\in H^2$ and the smooth initial data case, the solution of (\ref{model}) is in $H^4$ (see Figure \ref{fig:smooth mode}), which justifies the result in Lemma \ref{lm1p5}. In such case, the FD method still converges at the second order rate in time, which is surprising, but we note that similar super-convergence phenomenon of the FD method has also been reported in \cite{dnls-fd} and may need future analysis. Under the rough initial data case, i.e. $H^2$-data, the FD method suffers from the expected order reduction in time (see Figure \ref{fig:V2H2}).

3) The temporal error of the Strang splitting method shows wide oscillations and a unstable convergence rate, which significantly downgrades its computational efficiency. Under the smooth initial data case and the $H^2$-data case, the Strang splitting method has very close convergence results (see Figures \ref{fig:V2smooth}\&\ref{fig:V2H2}). This tells that the error of Strang splitting is mainly determined by the roughness from the potential.

4) For the spatial discretization error, under smooth data case, the three tested methods all show the second order convergence rate in space (see Figure \ref{fig:V2smoothspace}), where the dominate error is from the interpolation of the rough potential $\xi$ as we explained in Section \ref{sec:random}.

\begin{figure}[t!]
$$\begin{array}{cc}
\psfig{figure=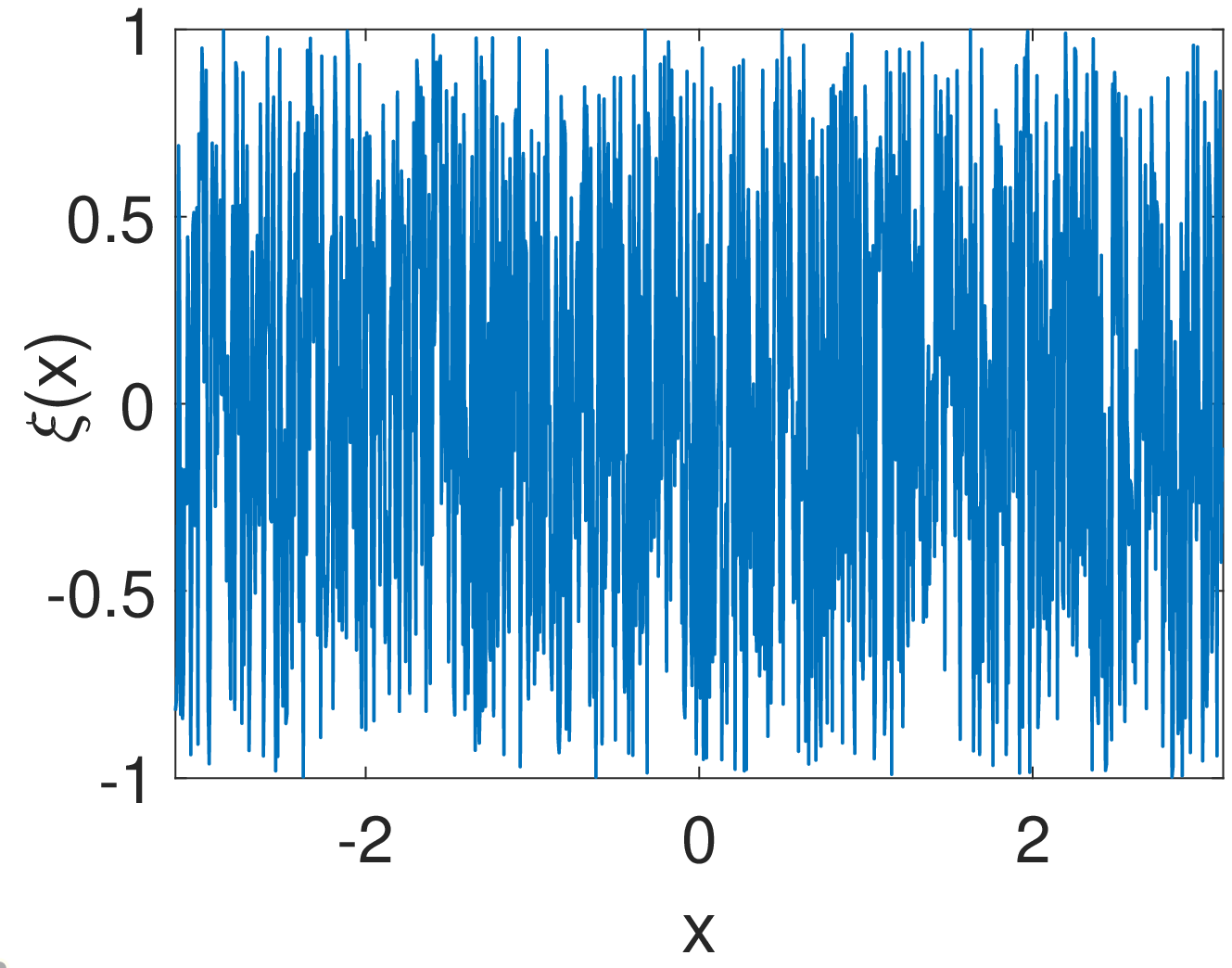,height=4.0cm,width=7cm}&
\psfig{figure=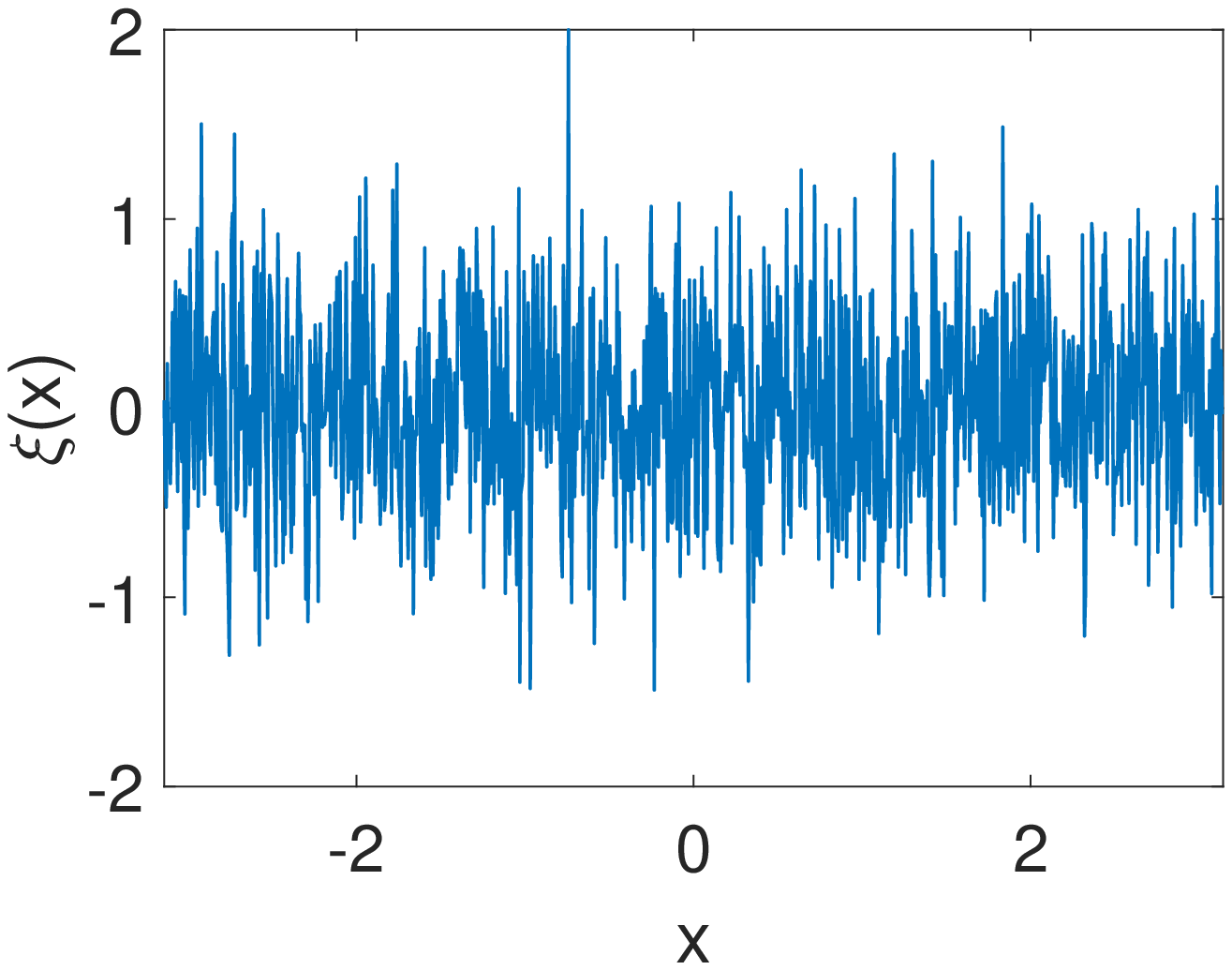,height=4.0cm,width=7cm}
\end{array}$$
\caption{Profile of a generated $\xi(x)$ in Example \ref{ex2}: pointwisely defined uniform distribution (left) or  (\ref{xi}) with $\theta=0$ under normal distribution (right).}
\label{fig:potential}
\end{figure}

\begin{figure}[t!]
$$\begin{array}{cc}
\psfig{figure=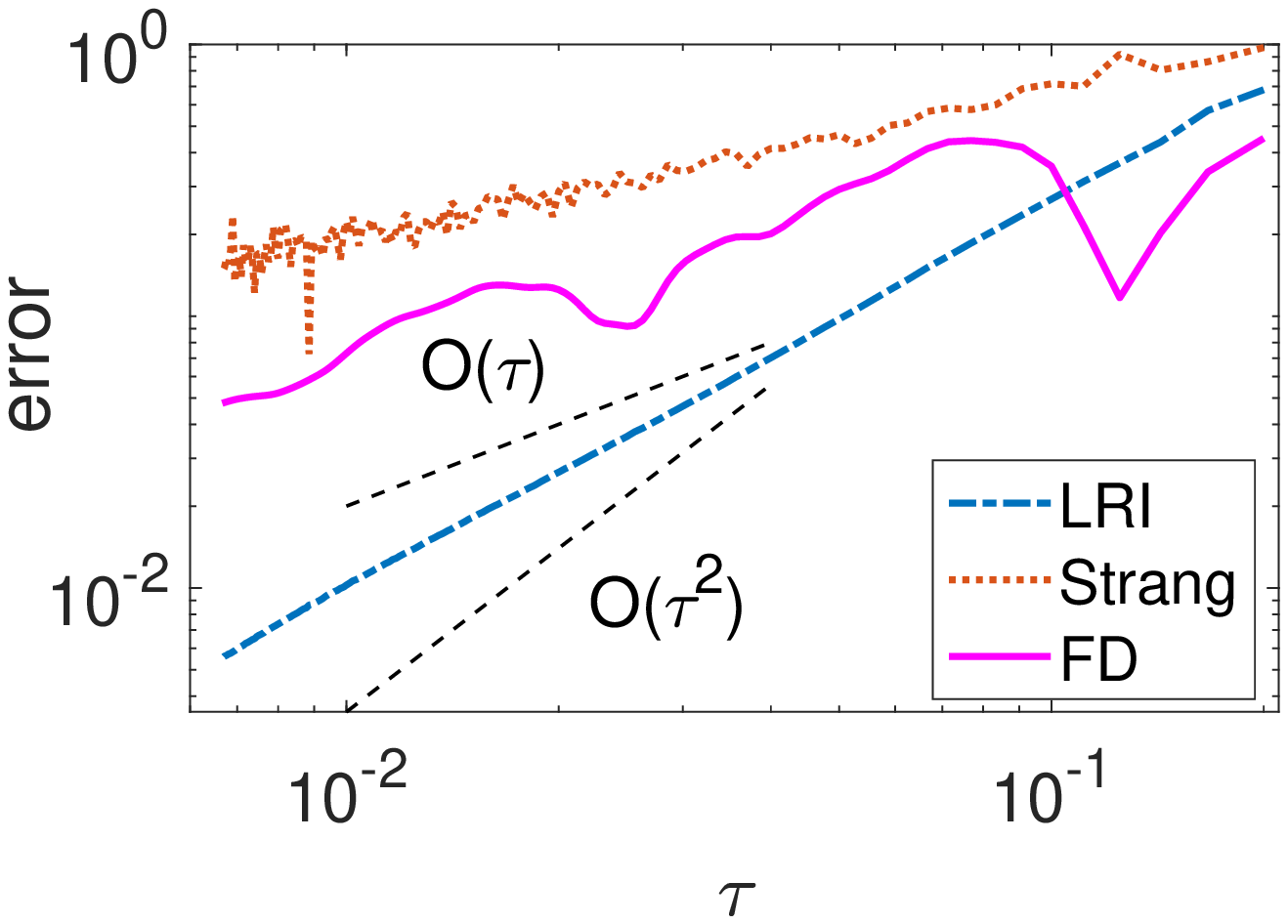,height=4.5cm,width=7cm}&
\psfig{figure=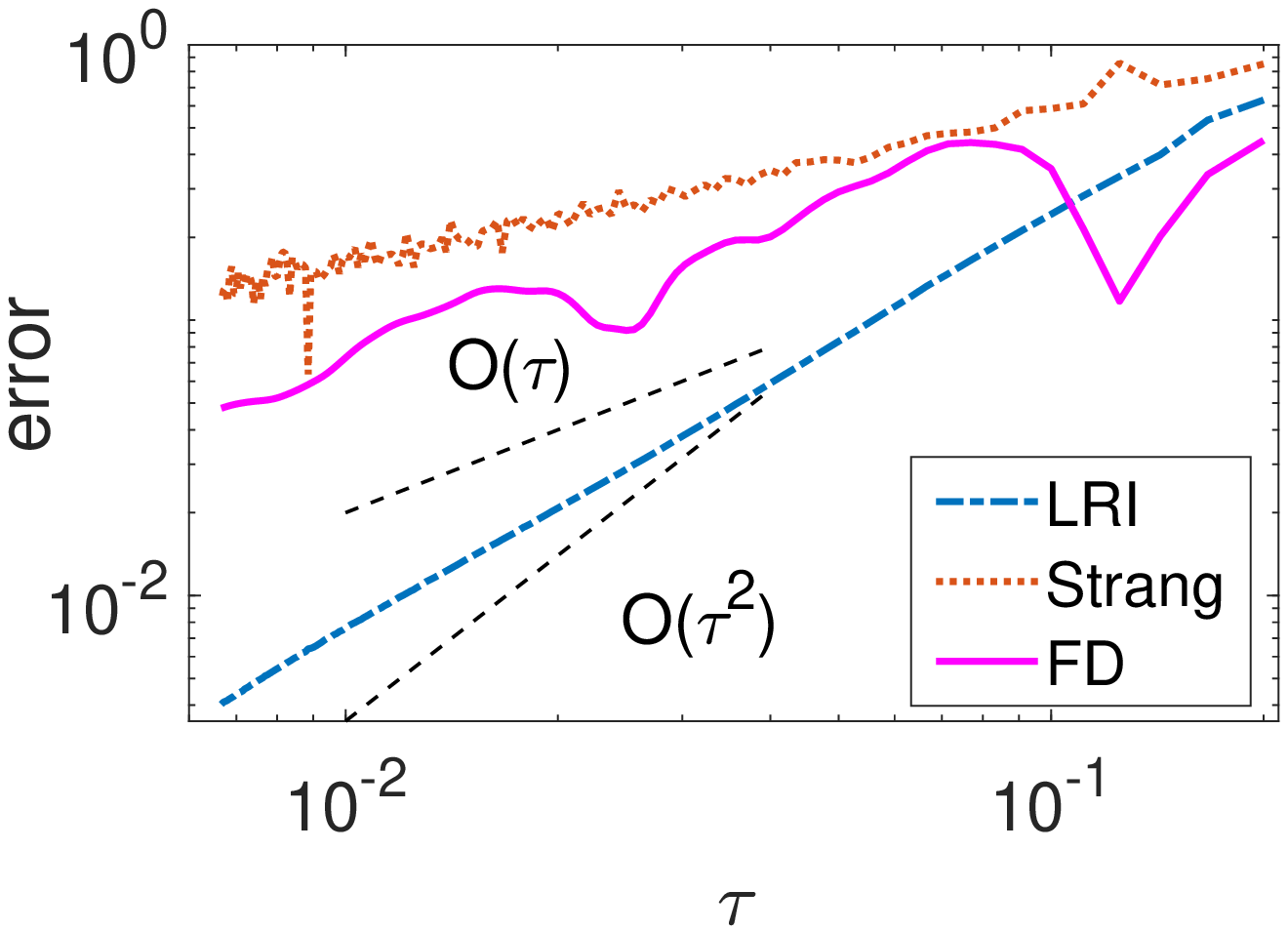,height=4.5cm,width=7cm}
\end{array}$$
\caption{Temporal convergence in Example \ref{ex2}: the error $\mathbb{E}(\|u-u^n\|_{L^2}/\|u\|_{L^2})$ under potential generated by pointwise uniform distribution (left) or (\ref{xi}) with normal distribution (right).}
\label{fig:1}
\end{figure}

\begin{figure}[t!]
$$\begin{array}{cc}
\psfig{figure=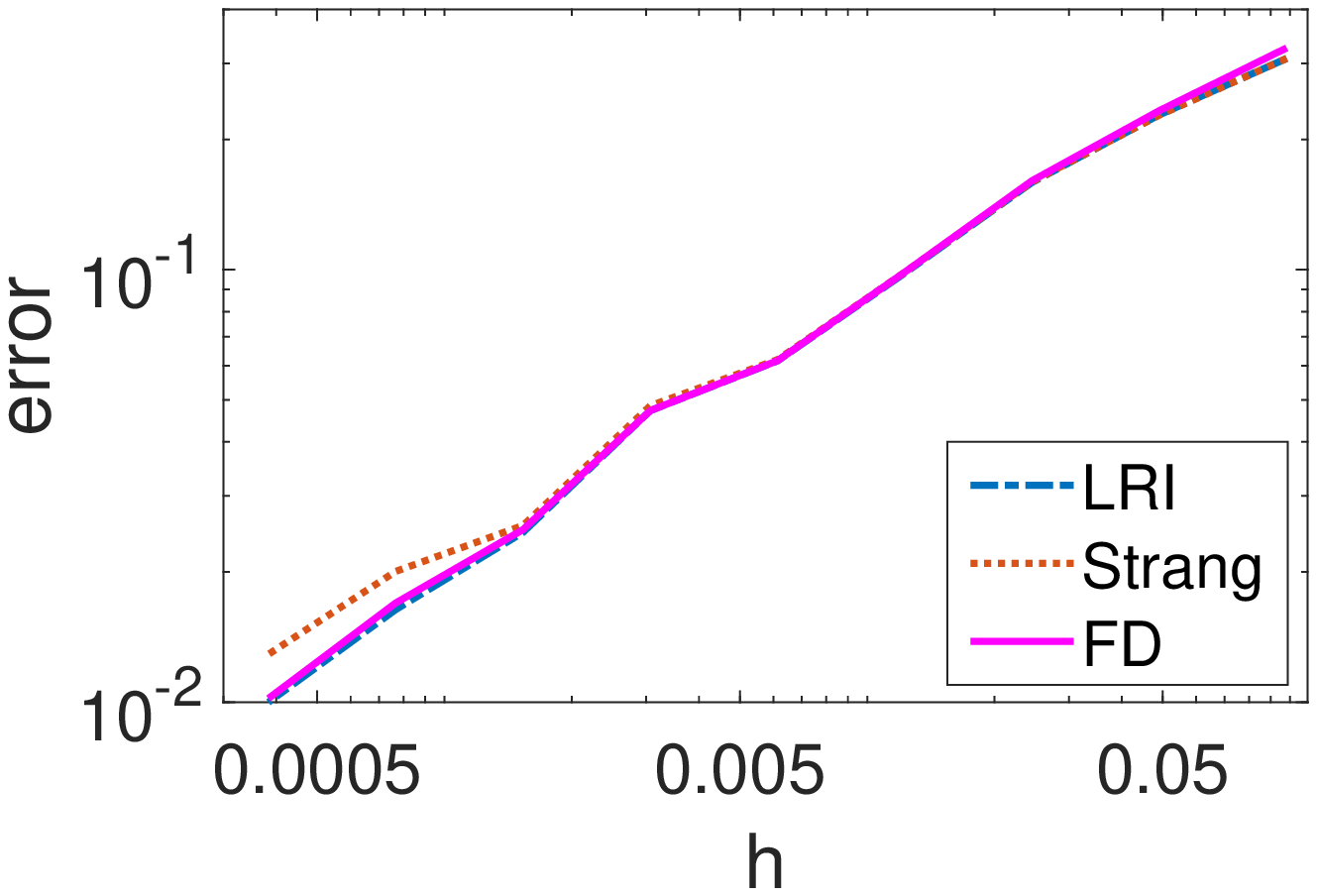,height=4.5cm,width=7cm}&
\psfig{figure=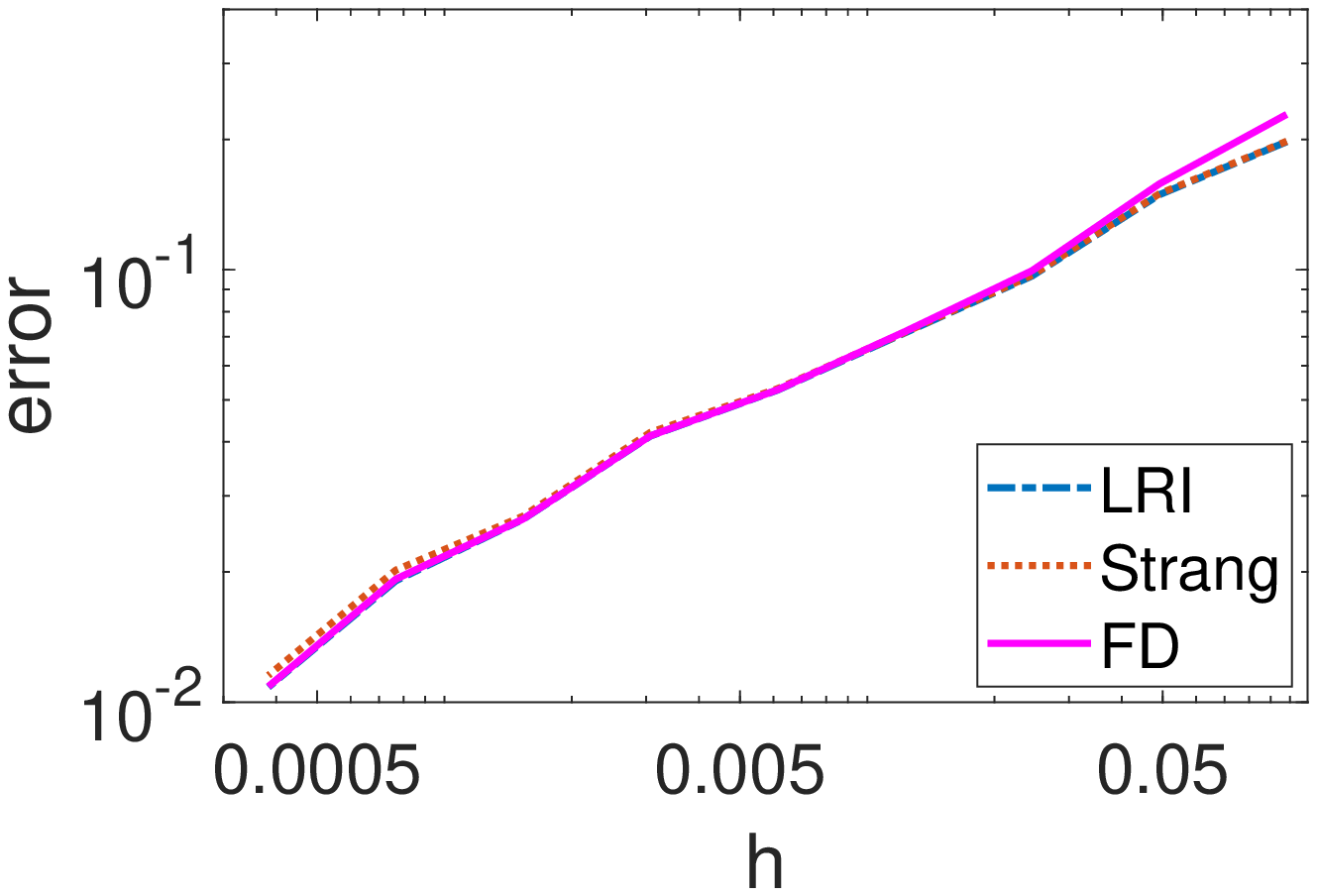,height=4.5cm,width=7cm}
\end{array}$$
\caption{Spatial convergence in Example \ref{ex2}: the error $\mathbb{E}(\|u-u^n\|_{L^2}/\|u\|_{L^2})$ under potential generated by pointwise uniform distribution (left) or (\ref{xi}) with normal distribution (right).}
\label{fig:space}
\end{figure}

\begin{figure}[t!]
$$\begin{array}{ccc}
\psfig{figure=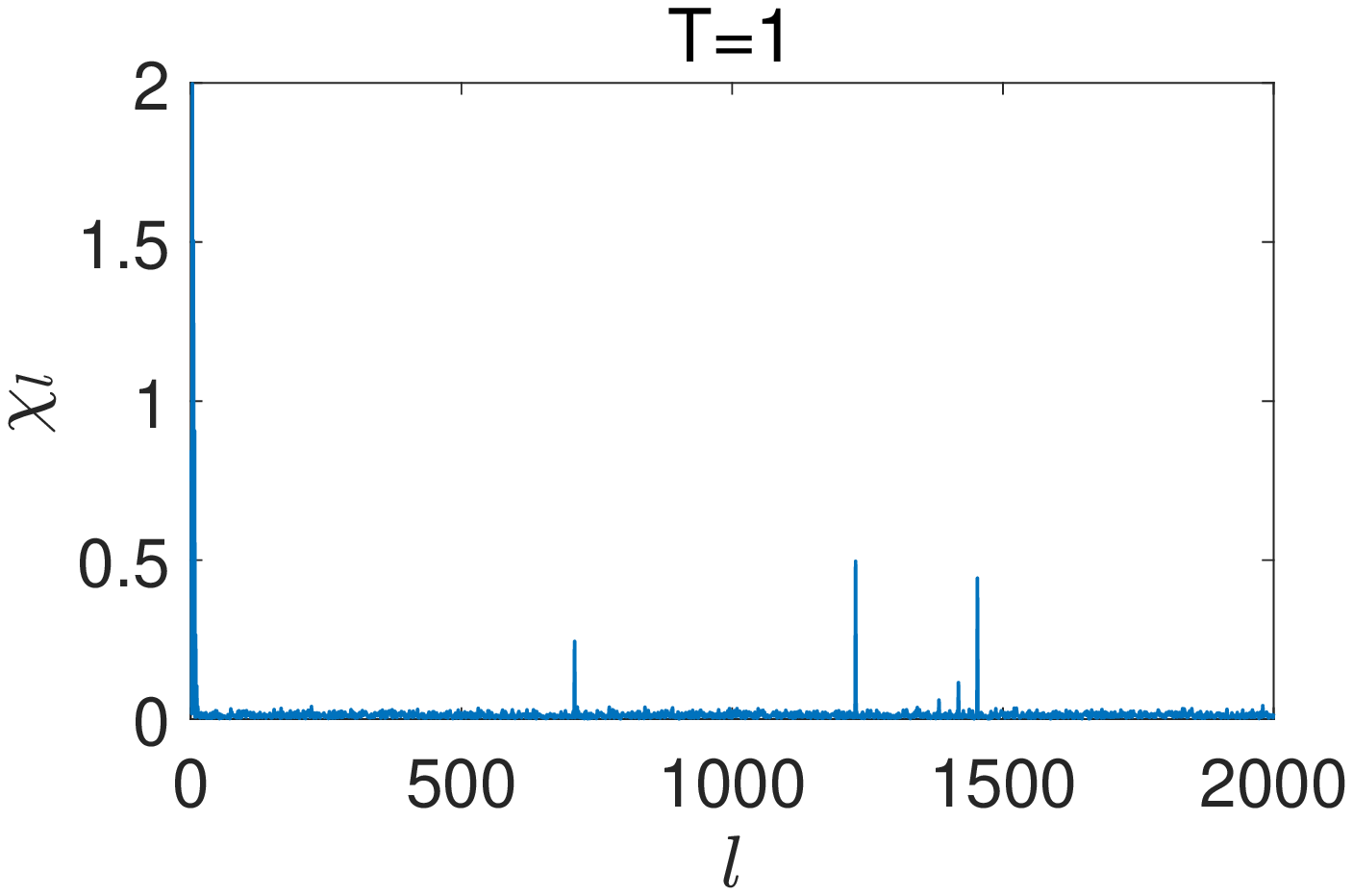,height=4.0cm,width=4.5cm}&
\psfig{figure=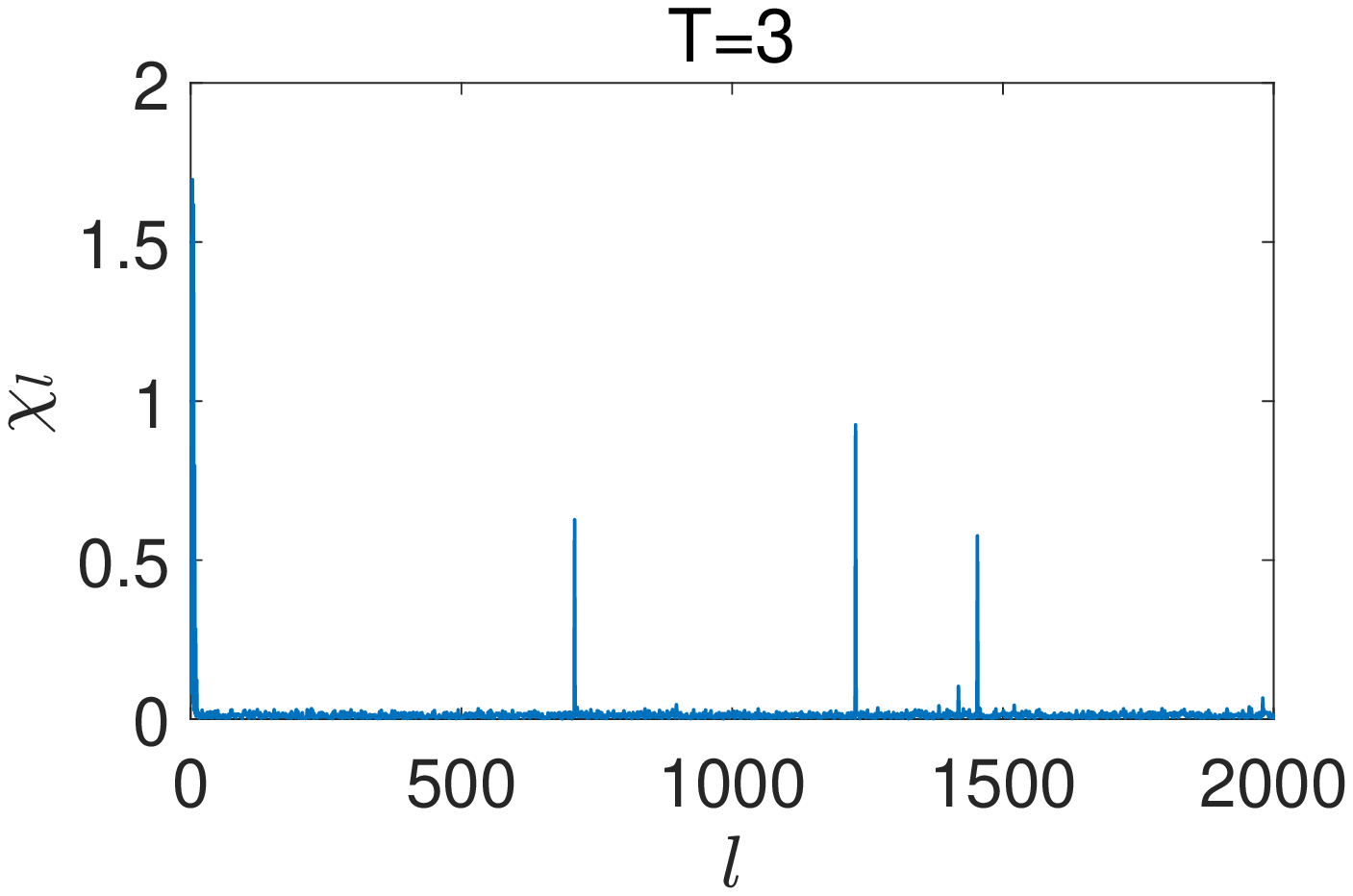,height=4.0cm,width=4.5cm}&
\psfig{figure=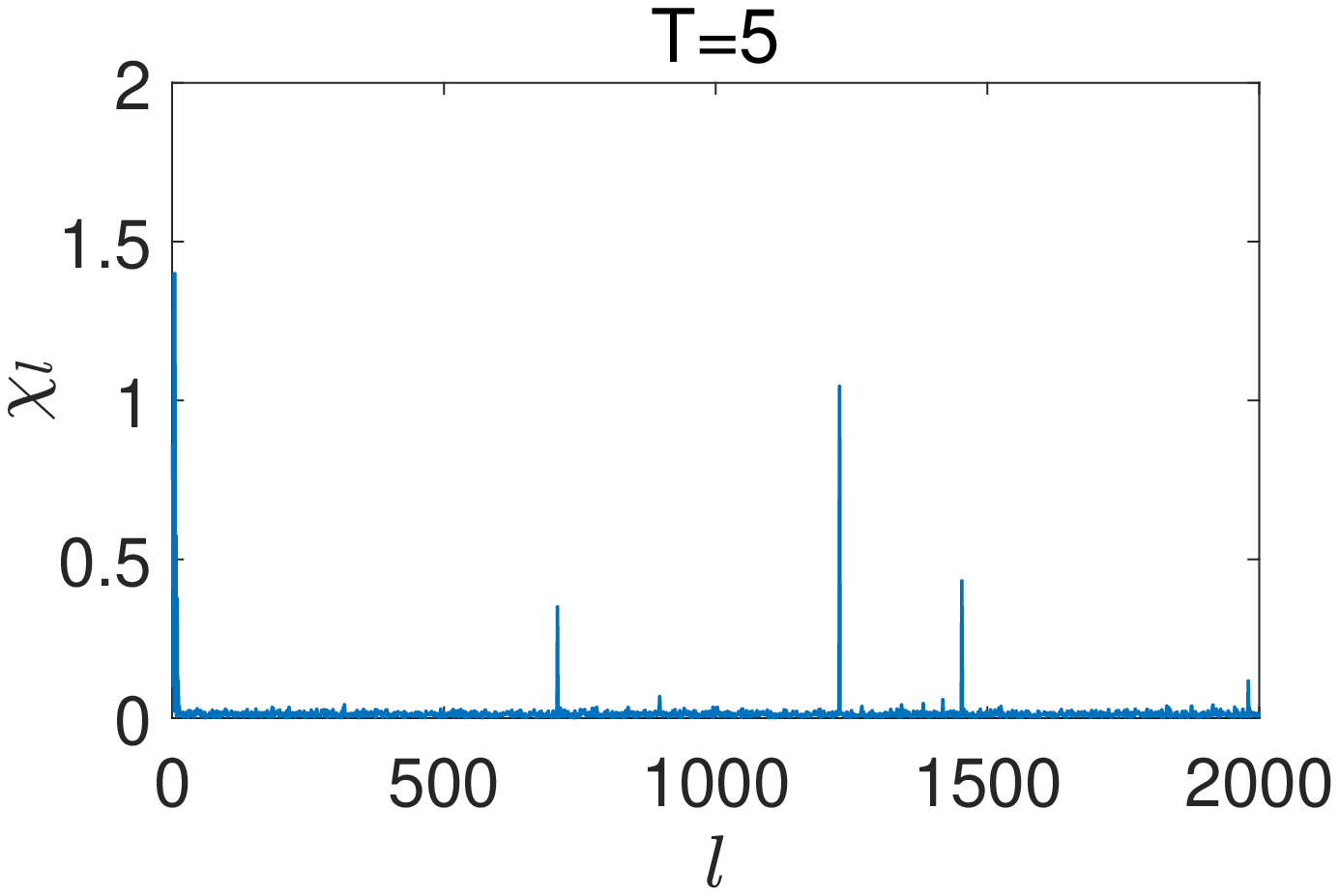,height=4.0cm,width=4.5cm}
\end{array}$$
\caption{Profile of scaled Fourier models $\chi_l(t)=\mathbb{E}(|\widehat{u}_l(t)|l^2)$ in Example \ref{ex2} at different time $t=T$ (potential generated by pointwise uniform  distribution).}
\label{fig:H2 mode}
\end{figure}

\begin{example}\label{ex2} ($L^2$-potential)
We now turn to test the presented numerical methods under the case of potentials  rougher than that can be covered in our convergence theorem. In (\ref{model}), we take $\lambda=1$ and consider potential $\xi\in L^2(\bT)$ with $\bT=(-\pi,\pi)$, which mimics the discrete case in (\ref{model2}) pointwisely:
\begin{equation}\label{pw xi}
  \xi(x_j)=\xi_j,\quad x_j=-\pi+jh,\quad j=0,1,\ldots,N-1,
\end{equation}
where $\xi_j$ are i.i.d. uniform distribution in $[-1,1]$ and $h=2\pi/N$ is the spatial mesh size with some integer $N>0$. We take a smooth initial data as
$$u_0(x)=\frac{2\cos(x)}{2+\sin(2x)}+i\cos(x),\quad x\in\bT.$$
Meanwhile, we also consider the kind of potential $\xi$ constructed from (\ref{xi}) with $\omega_1,\omega_2$ as the standard normal distribution, but now we take the regularizing parameter $\theta=0$.
As one sample of such two kinds of potentials, we use $N=2^{10}$ and the profiles of the $\xi(x)$ obtained as described are shown in Figure \ref{fig:potential}.
\end{example}

For convergence tests,
we take 100 samples of the random potential, and we solve (\ref{model}) for each one by the presented numerical methods till $t_n=T=2$. For each sample, the reference solution can be obtained either by the FD method (\ref{FD}) or by the LRI method (\ref{vduhamel3}) under very small mesh size, e.g.  $\tau=10^{-4}$ and the number of spatial grids $N=2^{16}$.
For each numerical method, again we compute the expectation of the relative error in $L^2$-norm (\ref{error def})
under different temporal or spatial mesh size.

To show the temporal accuracy, we fix the spatial mesh size $h=2\pi/N$ with $N=2^{16}$,  and the errors of the Strang splitting method (\ref{strang}), the FD method (\ref{FD}) and the LRI method (\ref{vduhamel3}) are plotted in Figure \ref{fig:1}. The corresponding spatial errors of the methods  are illustrated in Figure \ref{fig:space} under fixed time step $\tau=10^{-4}$.

At last, to illustrate the regularity of the solution in this example, again we consider the decay rate of the Fourier modes of the solution. We take 50 samples of the potential $\xi(x)$ generated by the pointwise uniform distribution (\ref{pw xi}) and compute the averaged scaled modes  $\chi_l(t):=\mathbb{E}(|\widehat{u}_l(t)|l^2)$. In Figure \ref{fig:H2 mode}, we plot the $\chi_l(t)$ for $0<l<N/2$ at different time $t$ obtained by the LRI method (\ref{vduhamel3}) under $N=2^{12}$ and $\tau=5\times 10^{-4}$.

Based on the numerical results in Figures \ref{fig:1}\&\ref{fig:space}, and other similar results not shown here for brevity, we have the following observations:

1) The temporal convergence rate of the Strang splitting method (\ref{strang}) and the FD method (\ref{FD}) are quite low (less than one) under the presence of the two kinds of very rough potentials (see Figure \ref{fig:1}). The LRI method converges in time with a rate between one and two, and it is much more accurate than the two classical methods.

2) For the space discretization, the error of all the methods are almost the same and they converge at a low rate (less than one) due to the very rough potential (see Figure \ref{fig:space}). Again, the convergence results of the two kinds of potentials are very close both in time and in space.

3) The solution in this example is at most in $H^2$ (see Figure \ref{fig:H2 mode}). The rougher potential, e.g. $\xi\in L^2$ makes the solution of (\ref{model}) less regular than that in Example \ref{ex1}.

\begin{figure}[t!]
$$\begin{array}{cc}
\psfig{figure=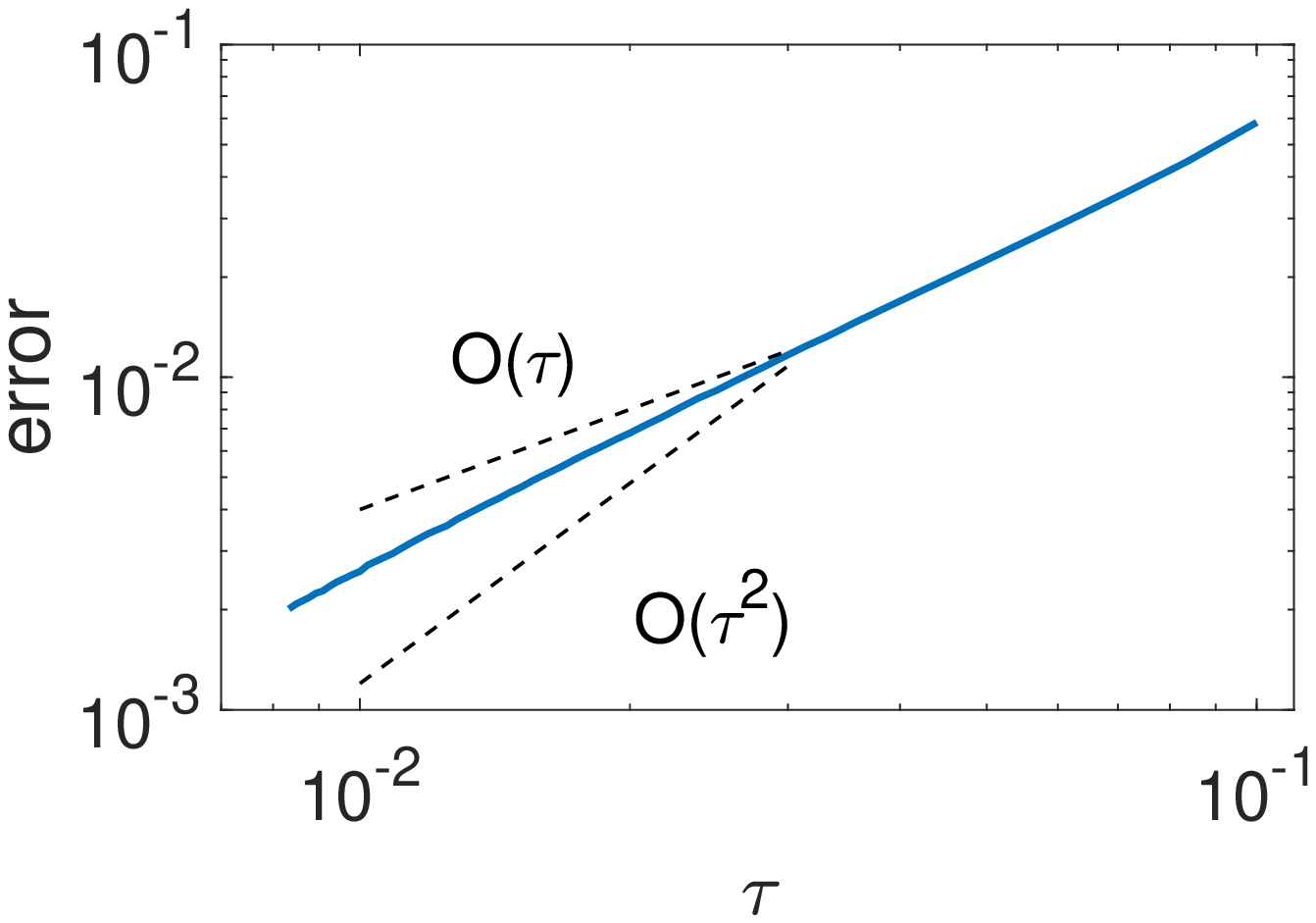,height=4.5cm,width=7cm}&
\psfig{figure=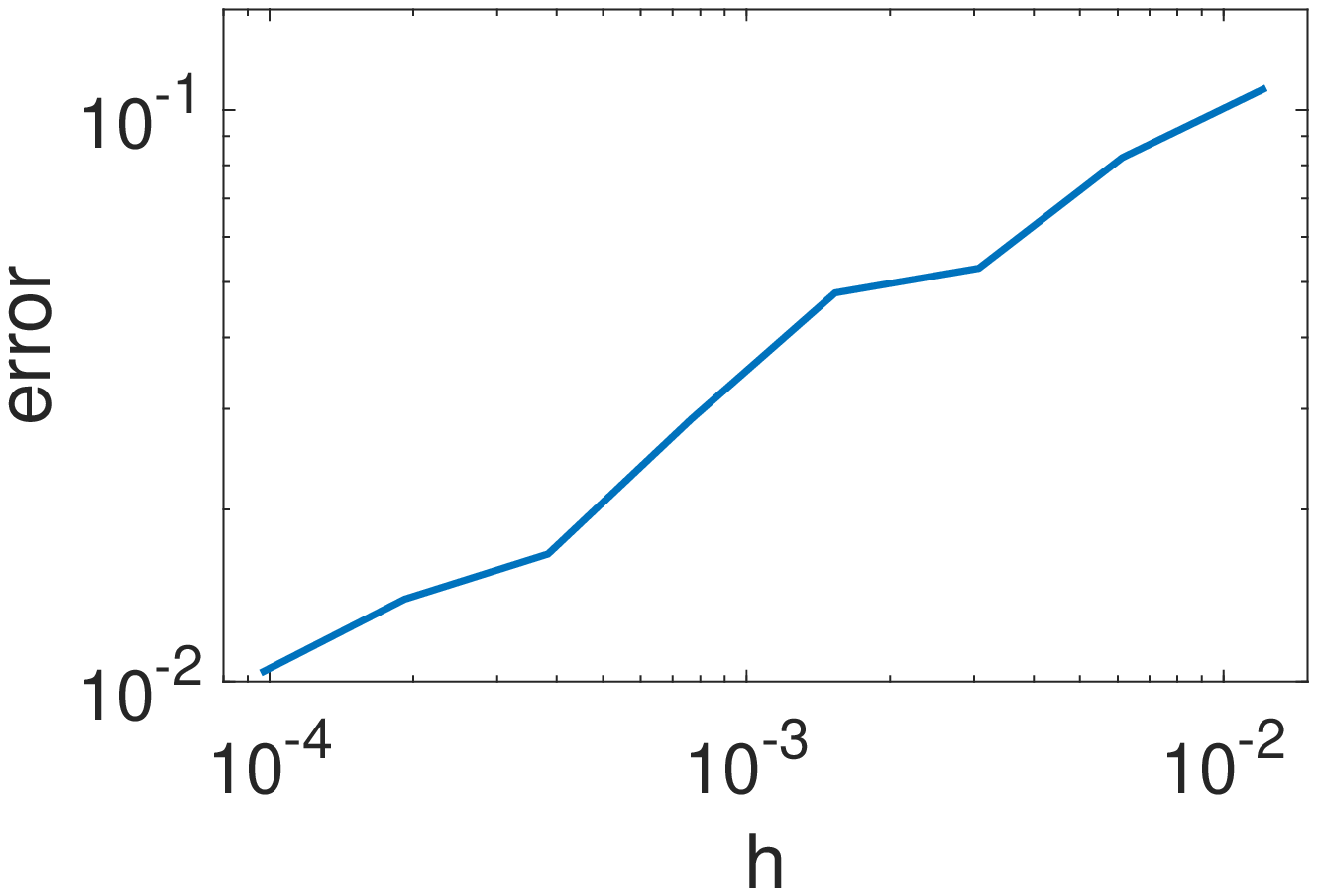,height=4.5cm,width=7cm}
\end{array}$$
\caption{Temporal convergence (left) and spatial convergence (right) of LRI in Example \ref{ex3}.}
\label{fig:whole}
\end{figure}
\begin{example}\label{ex3}(Whole space case)
The final example of the tests is devoted to  verifying the above numerical observations about the convergence of the  LRI scheme on a truncated whole space problem. In (\ref{model}), we take $\lambda=1$ and $\xi(x)$ the same as the pointwisely defined  uniform distribution in (\ref{pw xi}). The initial data is taken as
$$u_0(x)=2\mathrm{sech}(x^2),\quad x\in\bT,$$
where the domain is now chosen as $\bT=(-7\pi,7\pi)$ so that the boundary truncation error is negligible within the time of computation. We take 100 samples of $\xi$ as well, and the reference solution of each is obtained by LRI with $\tau=10^{-4}$ and $N=2^{18}$. The temporal error and the spatial error (\ref{error def}) of the LRI method (\ref{vduhamel3}) at $t_n=T=1$ are plotted in Figure \ref{fig:whole}. The convergence results of the two classical methods are similar as in Example \ref{ex2}, and so they are omitted here for brevity.
\end{example}

As can be seen in Figure \ref{fig:whole}, the numerical observations are totally the same as in Example \ref{ex2} for the temporal and the spatial convergence of the LRI method (\ref{vduhamel3}).

\begin{figure}[t!]
$$\begin{array}{cc}
\psfig{figure=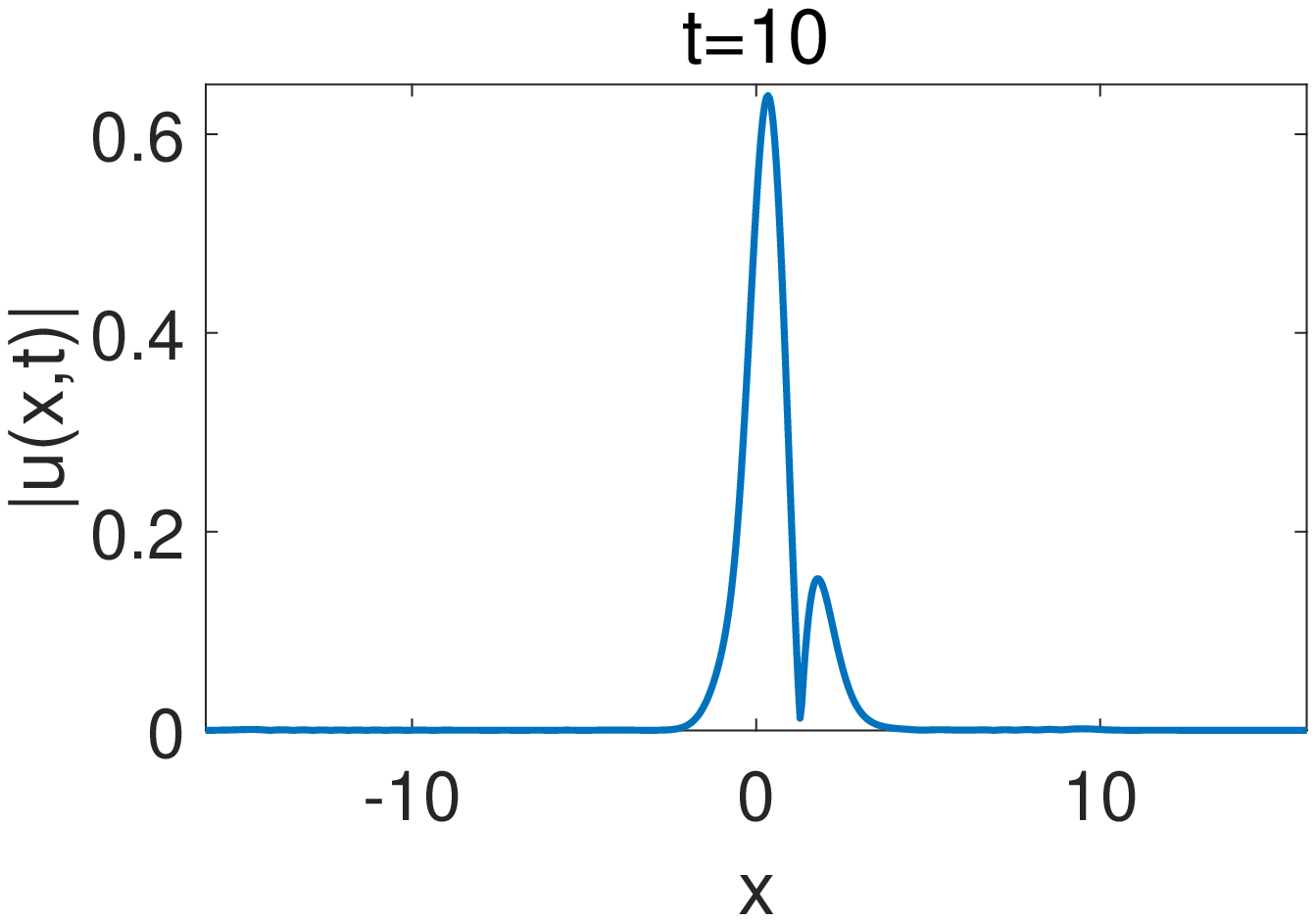,height=4.5cm,width=7cm}&
\psfig{figure=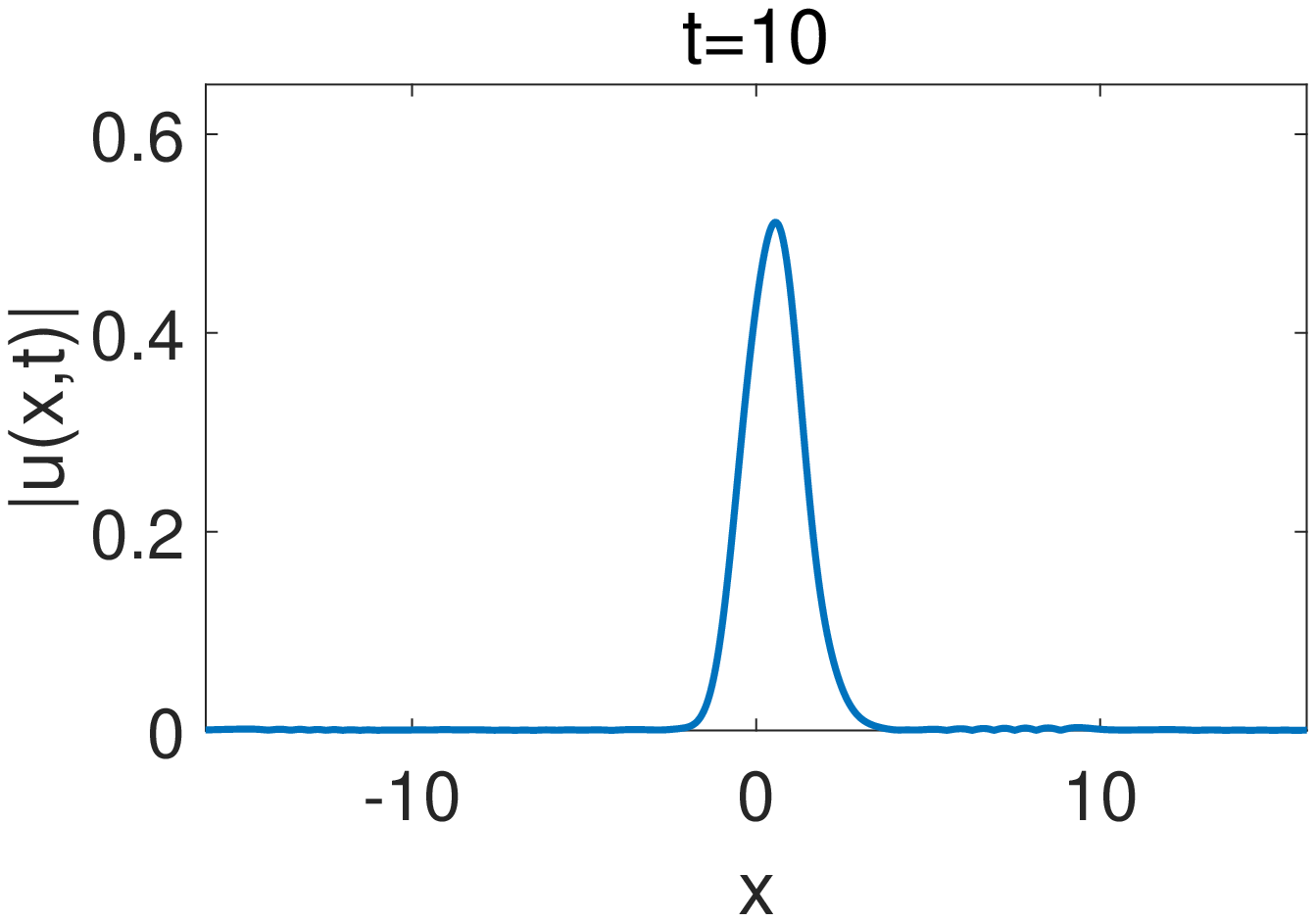,height=4.5cm,width=7cm}\\
\psfig{figure=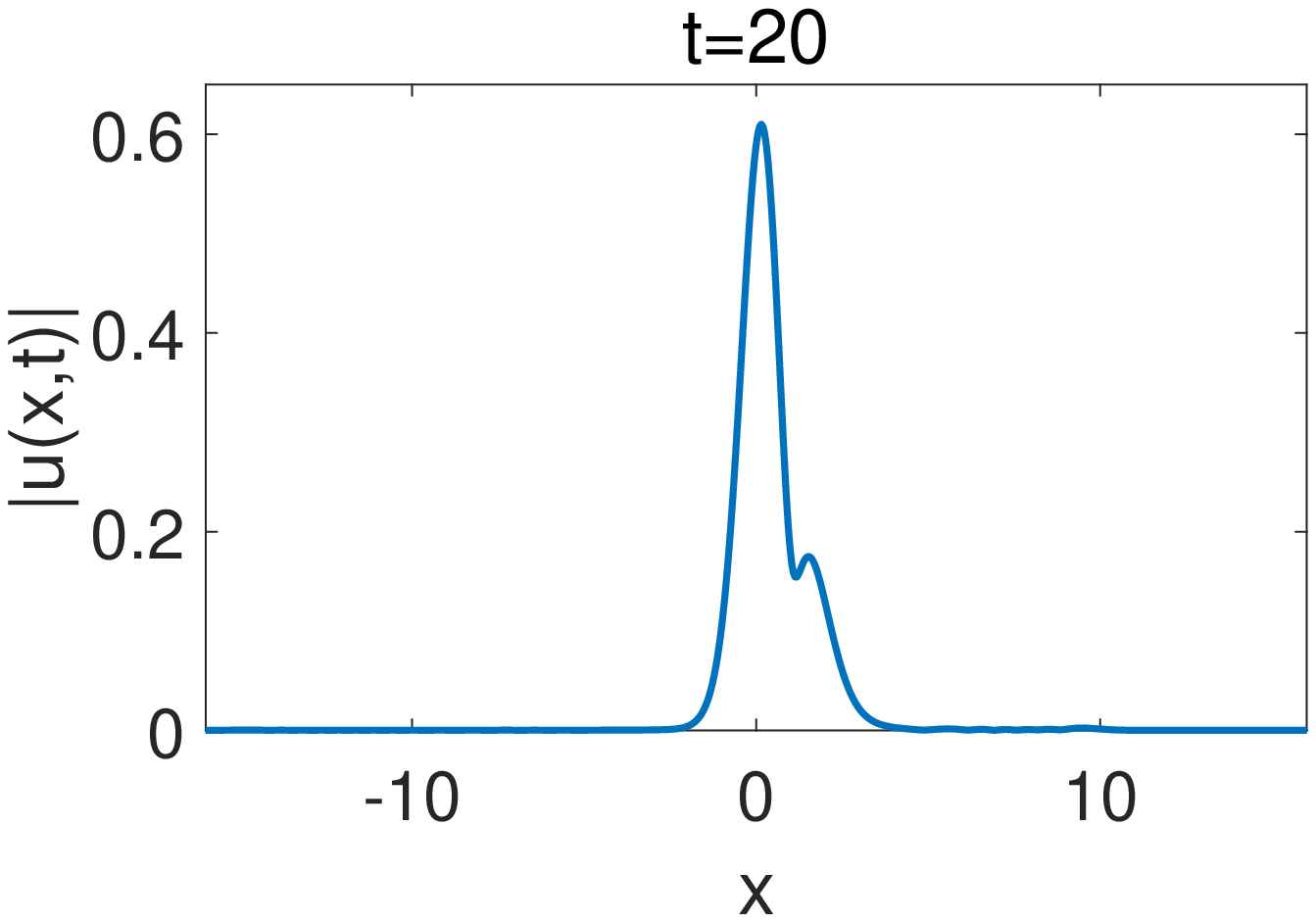,height=4.5cm,width=7cm}&
\psfig{figure=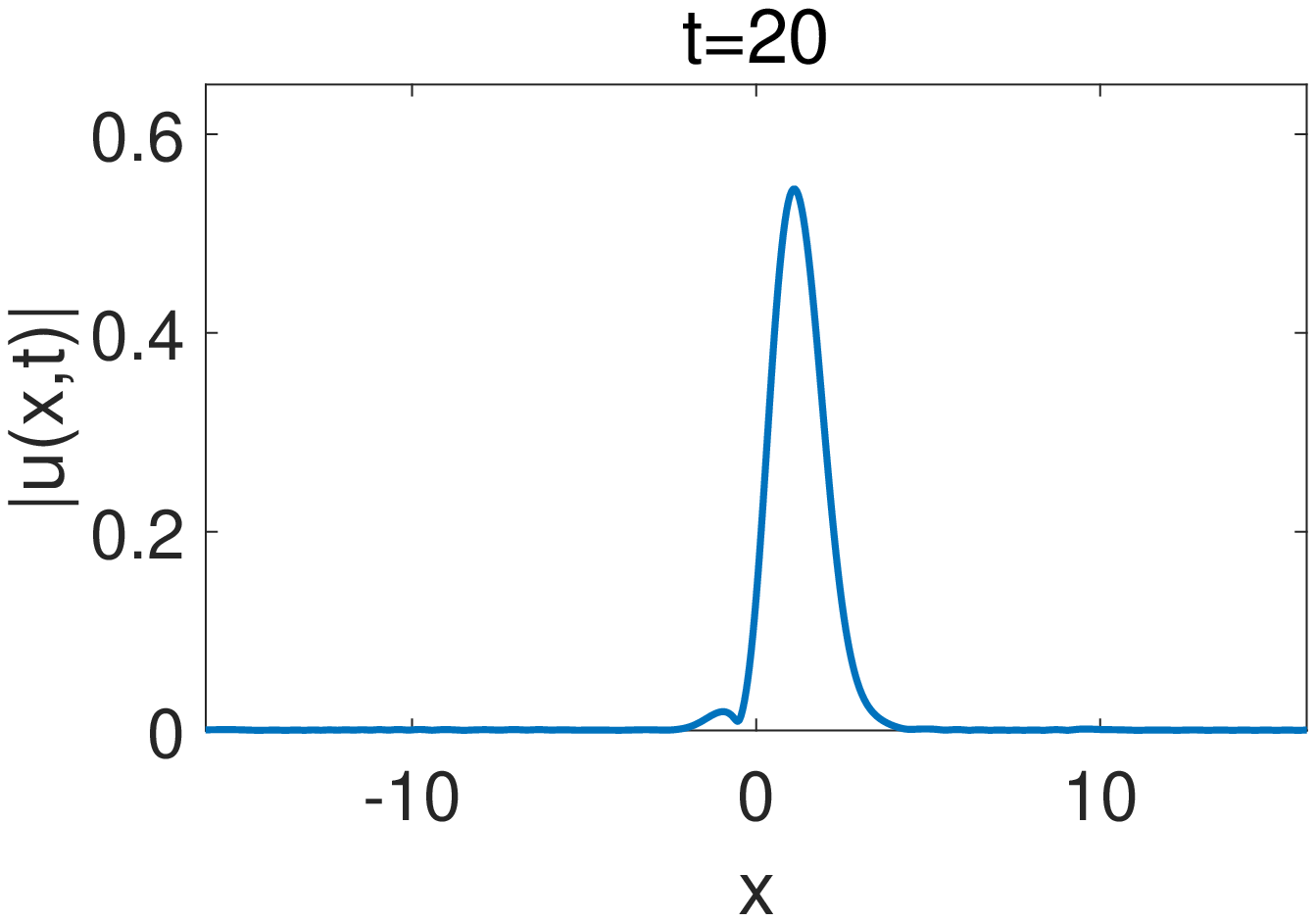,height=4.5cm,width=7cm}
\end{array}$$
\caption{Profile of $|u(x,t)|$ in Section \ref{sec:last num} at $t=10$ and $t=20$ under the linear case (left) and the nonlinear case (right).}
\label{fig:dynamics}
\end{figure}

\subsection{A simulation of localization/spreading}\label{sec:last num} Now, by applying the LRI method, we perform a numerical experiment to show the dynamics of an initially localized wave in (\ref{model}). We take the initial data
$$u_0(x)=\frac{1}{\sqrt{\pi}}\fe^{-x^2},\quad x\in \bT,$$
with the domain $\bT=(-256\pi,256\pi)$. The number of spatial grids is $N=2^{15}$, and the random potential (one sample) is constructed as
$$\xi(x)=\frac{13}{\sqrt{N_0}}
\sum_{l=-N_0/2}^{N_0/2-1}\left[\mathrm{rand}(N_0,1)+i\,\mathrm{rand}(N_0,1)\right]
\fe^{i\mu_l(x+L)}+c.c.,\quad x\in\bT,$$
with $N_0=N/64$. We use the  LRI scheme (\ref{vduhamel3}) with time step $\tau=0.0025$ to solve (\ref{model}) till $T=20$. The Figure \ref{fig:dynamics} shows the solution $|u(x,t)|$ at different $t$ under the linear case ($\lambda=0$) and the nonlinear case ($\lambda=10$).

We can see that in Figure \ref{fig:dynamics}, the usual spreading of waves in the Schr\"odinger equation is suppressed by the random potential, and the solutions of (\ref{model}) in both linear and nonlinear cases are localized up to the time of simulation in this example. We also remark that in our simulation of the nonlinear case, although the main part of the waves are localized as in the Figure \ref{fig:dynamics}, we observed some tiny radiations towards the boundary, and this motivates us to use the large size for the domain.
Note that the possible spreading or the so-called weak turbulence in the D-NLS (\ref{model}) could take place at a very large time, e.g. $t=10^{8}$ as showed for the discrete model (\ref{discrete}) in \cite{prl2,prl1}. The accurate numerical simulation of (\ref{model}) up to such scale of time requests much more future efforts to address problems like boundary truncations and the long-time behaviour of the scheme.

\section{Conclusion}\label{sec:con}
In this work, we have addressed the accuracy problem of the numerical integrators for solving the continuous disordered nonlinear Schr\"odinger equation with a spatial random potential on an one-dimensional torus. The random potential can be very rough, which introduces roughness also to the solution. We showed that classical numerical  integrators suffer from severe accuracy order reductions due to the presence of the rough potential, and their approximations are much less accurate than in the usual smooth setup.  By applying a low-regularity Fourier integrator (LRI) to this problem, where particular efforts were made to integrate the potential term, we showed that the finite time accuracy can be raised,  compared to classical methods. In particular, a second order temporal convergence rate of LRI in $L^2$-norm was established rigorously for potentials in $H^2$. For potentials in $L^2$ and some other popular examples  used in the applications of disordered problem, a convergence rate of LRI between one and two were observed.  Many numerical experiments were provided to illustrate the performance of the schemes with comparisons made under different types of the random potentials.

\appendix
\section{Schemes for the discrete model}\label{sec:app}

The splitting for (\ref{model2}) reads
\begin{equation}\label{sub11}
\Phi_T^t:\quad \left\{
\begin{split}
&i\dot{v}_l(t)=-J[ v_{l+1}(t)+v_{l-1}(t)],\quad  t>0,\ -N\leq l\leq N,\\
&v_l(0)=v_l^0, \quad -N\leq l\leq N,
\end{split}\right.
\end{equation}
and
\begin{equation}\label{sub22}
\Phi_V^t:\quad\left\{
\begin{split}
&i\dot{w}_l(t)=
\left(\xi_l+\lambda|w_l(t)|^2\right)w_l(t), \quad t>0,\ -N\leq l\leq N,\\
&w_l(0)=w_l^0,\quad -N\leq l\leq N.
\end{split}\right.
\end{equation}
For $\Phi_T^t$, by taking Fourier transform
$$\widehat{v}_j(t)=\frac{1}{2N}\sum_{l=-N}^{N-1}\fe^{-ijl\pi/N}v_l(t),\ -N\leq j<N,\quad
v_l(t)=\sum_{j=-N}^{N-1}\fe^{ijl\pi/N}\widehat{v}_j(t),\ -N\leq l<N,$$
the equation in (\ref{sub11}) is diagonalized and the exact solution is
$$v_l(t)=\sum_{j=-N}^{N-1}\fe^{2itJ\cos(j\pi /N)}
\fe^{ijl\pi/N}\widehat{(v^0)}_j, \quad t\geq0,\ -N\leq l\leq N,$$
with $v^0=(v^0_{-N},\ldots,v^0_{N})$. For $\Phi_T^t$, noting that $|w_l(t)|$ is constant in $t$ for all $l$ in (\ref{sub22}), so the exact solution is
$$w_l(t)=\fe^{-it(\xi_l+\lambda|w_l^0|^2)}w_l^0,\quad t\geq0,\ -N\leq l\leq N.$$
Then the Strang splitting scheme can be written down same as the composition in (\ref{strang}).

The corresponding semi-implicit finite difference method is simply
\begin{align*}
 &i\frac{u_l^{n+1}-u_l^{n-1}}{2\tau}=-\frac{J}{2}
\left(u_{l+1}^{n+1}+u_{l-1}^{n+1}+u_{l+1}^{n-1}+u_{l-1}^{n-1}\right)
+\xi_lu_l^{n}+\lambda|u_l^n|^2u_l^n,\quad n\geq1,
\end{align*}
for $-N\leq l<N$ with
\begin{align*}
&u_l^1=u^0_l+iJ\tau\left(u^0_{l+1}+u^0_{l-1}\right)
-i\tau\left(\xi_l+\lambda|u^0_l|^2\right)u^0_l,\quad -N\leq l\leq N,\\
&u^n_{-N}=u^n_{N},\quad u^n_{-N-1}=u^n_{N-1},\quad n\geq0.
\end{align*}

\section*{Acknowledgements}
X. Zhao is partially supported by the NSFC 11901440 and the Natural Science Foundation of Hubei Province No. 2019CFA007. We thank Prof. Avy Soffer for introducing the disordered problem and thank Prof. Katharina Schratz for communications on the numerical method.

\bibliographystyle{model1-num-names}

\end{document}